\documentclass[12pt,fleqn]{article}
\usepackage[english]{babel}
\usepackage{makeidx}
\usepackage{latexsym,amsfonts,amssymb,amsmath,longtable,amsthm}
\input amssym.def
\usepackage{latexsym}
\usepackage[all]{xy}
\usepackage{amsfonts}
\usepackage{amsmath}
\usepackage{mathrsfs}
\usepackage{color}
\usepackage{ulem}
\usepackage[table]{xcolor}
\usepackage[unicode,breaklinks=true,colorlinks=true]{hyperref}
\catcode`@=11
\newbox\tr@tto
\setbox\tr@tto=\hbox{{\count0=0\dimen0=-
,9pt\dimen1=1,1pt\loop\ifnum\count0<11 \advance \count0 by1 \vrule
width.51pt height\dimen1
     depth\dimen0\kern-0.17pt\advance\dimen0 by-
0.05pt\advance\dimen1
     by0.1pt\repeat \loop\ifnum\count0<21\advance \count0 by1
\vrule
     width.6pt height\dimen1 depth\dimen0\kern-0.2pt
\advance\dimen0
     by-0.1pt\advance\dimen1 by 0.05pt\repeat}}
\def\medint{\displaystyle\copy\tr@tto\kern-10.4pt\int}
\catcode`@=12

\def\Xint#1{\mathchoice
   {\XXint\displaystyle\textstyle{#1}}%
   {\XXint\textstyle\scriptstyle{#1}}%
   {\XXint\scriptstyle\scriptscriptstyle{#1}}%
   {\XXint\scriptscriptstyle\scriptscriptstyle{#1}}%
   \!\int}
\def\XXint#1#2#3{{\setbox0=\hbox{$#1{#2#3}{\int}$}
     \vcenter{\hbox{$#2#3$}}\kern-.5\wd0}}

\def\dashint{\Xint-}

\newcommand{\R}{{\mathbb R}}
\renewcommand{\C}{{\mathbb C}}

\newcommand{\N}{{\mathbb N}}

\newcommand{\e}{{\varepsilon}}
\newcommand{\ep}{{\epsilon}}
\renewcommand{\th}{{\theta}}
\newcommand{\al}{{\alpha}}
\newcommand{\Le}{{\mathscr L}}
\newcommand{\GG}{{\mathscr G}}

\newcommand{\Be}{{\mathcal  B}}
\newcommand{\Me}{{\mathscr  M}}
\renewcommand{\H}{{\mathcal H}}
\newcommand{\rank}{{\mathrm{rank\,}}}

\newcommand{\p}{{p_\circ}}

\renewcommand{\b}{{q_\circ}}

\renewcommand{\t}{\tau}
\renewcommand{\tt}{\tau_*}

\newcommand{\loc}{{\rm loc}}

\newcommand{\meas}{\mathop{\rm meas}}
\newcommand{\diam}{\mathop{\rm diam}}

\newcommand{\LL}{\mathrm{L}}
\newcommand{\WW}{\mathrm{W}}

\newcommand{\CC}{\mathrm{C}}

\newcommand{\Cc}{C^{k,\alpha}}
\newcommand{\BBB}{B^{k+\al}_{p,\infty}(\R^n,\R^d)}

\newcommand{\LK}{\Le^{\al}_{p}}
\newcommand{\LLL}{\Le^{\al}_{p,1}}

\newtheorem{ttt}{\bf Theorem}[section]

\newtheorem{lem}{\bf Lemma}[section]
\newtheorem{cor}{\bf Corollary}[section]
\newtheorem{df}{\bf Definition}[section]
\theoremstyle{remark}
\newtheorem{rem}{\bf Remark}[section]

\numberwithin{equation}{section}

\newcommand{\dd}{{\rm d}}

\title{ON LUZIN $N$-PROPERTY  AND UNCERTAINTY PRINCIPLE FOR THE SOBOLEV MAPPINGS}
\author{Adele Ferone, \, Mikhail V.~Korobkov,  \, and \, Alba Roviello}

\begin{document}

\maketitle

\begin{abstract}
We say that a mapping $v:\R^n\to\R^d$ satisfies the $(\tau,\sigma)$--N--property, 
if $\H^\sigma(v(E))=0$ whenever $\H^\tau(E)=0$, where $\H^\tau$ means the Hausdorff measure. 
We prove that every mapping~$v$ of Sobolev class $W^k_p(\R^n,\R^d)$ with $kp>n$ satisfies $(\tau,\sigma)$--N--property for every  $0<\tau\ne\tau_*:=n-(\alpha-1) p$
with $$\sigma=\sigma(\tau):=\left\{\begin{array}{lcr}\tau, \ \ &\mbox{\rm if \ }& \tau>\tau_*;\\ [13pt]
\frac{p\,\tau}{\alpha p-n+\tau}, \ \ &\mbox{\rm if \ }& 0<\tau <\tau_*.\end{array}\right.$$

We prove also, that for $k>1$ and for the critical value $\tau=\tau_*$ the corresponding $(\tau,\sigma)$--N--property fails in general. 
Nevertheless, this $(\tau,\sigma)$--N--property holds for $\tau=\tau_*$ if we assume in addition that the highest derivatives $\nabla^k v$ belong to the 
Lorentz space $L_{p,1}(\R^n)$ instead of~$L_p$. 

We extend these results to the case of fractional Sobolev spaces and for the Besov spaces as well. 
Also, we establish  some Fubini type theorems for $N$-properties and discuss their applications to the Morse--Sard theorem and its recent extensions.

\medskip

\noindent {\bf Key words:} {\it Sobolev--Lorentz mappings, fractional Sobolev classes, Luzin $N$--property, Morse--Sard theorem,
Hausdorff measure}
\end{abstract}


\section{Introduction}\label{Introd}

The classical Luzin $N$-property means that for a mapping $f:\R^n\to\R^n$ one has $\meas f(E)=0$ whenever $\meas E=0$. (Here 
$\meas E$ is the usual $n$-dimensional Lebesgue measure.) 

This property plays one of the crucial roles  in the classical real analysis and differentiation theory (see, e.g., \cite{Saks}\,). 
It is very useful also in elasticity theory and in geometrical analysis, especially in the theory of quasiconformal mappings and, more generally, in the theory of 
mappings with bounded distortions, i.e., mappings $f:\Omega\subset\R^n\to\R^n$ of Sobolev class $W^1_n(\R^n)$ 
such that $\|f'(x)\|^n\le K\det f'(x)$ almost everywhere with some constant $K\in[1,+\infty)$. 
The notion of mappings with bounded distortion was introduced by Yu.G.~Reshetnyak (see, e.g., his classical books
\cite{Resh1}, \cite{Resh2}, \cite{GR}\,). He proved that they satisfy the $N$-property and 
this was very helpful in his subsequent proofs of other basic topological 
properties of such mappings (openness, discreteness and etc.). 

The notion of mappings with bounded distortion leads to the theory of more general mappings with finite distortion 
(i.e., when $K$ in the above definition could depend on~$x$  and is not assumed  
to be uniformly bounded; see, e.g., one of the~pioneering papers \cite{VodG} where the monotonicity, continuity and $N$--property of such mappings from the~class~$W^1_n$ were established).  
This theory has been 
intensively developed in the last decades (see, e.g., the book~\cite{HenclBook} for the actual overviews), and studying the~$N$-property constitutes one of the most important direction (see, e.g., 
\cite{KauphMich} and the more recent paper \cite{Onofrio2016}). 

Note that the belonging of a mapping to the Sobolev class $W^1_n(\R^n,\R^n)$ is crucial for $N$-properties. 
Indeed, every mapping of class $W^1_p(\R^n,\R^n)$ with $p>n$ is continuous and supports the $N$-property (it is a simple consequence of the Morrey inequality). But even if a mapping 
$f\in W^1_n(\R^n,\R^n)$ is continuous (which is not guarantied in general), it may have no $N$-property. On the other hand, the~$N$-property holds for functions 
of the~class  $W^1_n(\R^n,\R^n)$ under some additional  assumptions on its topological features, namely, for homeomorphic and open mappings (see~\cite{Resh-op}), for quasi-monotone\footnote{Some of these results were generalised for the more delicate case of Carnot groups and manifolds, see, e.g.,~\cite{Vod-N}.} mappings (see \cite{VodG}, \cite{MalyMartio}\,), etc. 

The above results are very delicate and sharp: indeed, for any $p<n$ there are homeomorphisms $f\in W^1_p(\R^n,\R^n)$ without~$N$-property. This phenomena was discovered by S.P.~Ponomarev \cite{Pon1}.
In the last years his construction has been very refined and it was constructed an example of Sobolev homeomorphism with zero Jacobian a.e. which belongs simultaneously to all the 
classes $W^1_p(\R^n,\R^n)$ with $p<n$  (\,see \cite{Hencl2011} and \cite{Cherny2011}\,)~--- of course,
this "strange" homeomorphism certainly fails to have the~$N$-property\footnote{Moreover, even the examples of bi-Sobolev homeomorphisms of class
$W^1_p(\R^n,\R^n)$, $p<n-1$, with zero Jacobian a.e. were constructed recently, see, 
e.g., \cite{HenclOnofrio2014}, \cite{Cherny2015}. Such homeomorphisms are impossible in the Sobolev class~$W^1_{n-1}(\R^n,\R^n)$.
Furthermore, in \cite{Hencl16-Arch} authors constructed an example of Sobolev homeomorphism $f\in W^1_1((0,1)^n,\R^n)$ such that the Jacobian \,$\det f'(x)$\, changes 
its sign on the sets of positive measures. 
}.  

In the positive direction, it was proved in~\cite{KauhMaly} (see also~\cite{rom1}\,), that every mapping of the Sobolev--Lorentz class $W^1_{n,1}(\R^n,\R^n)$ (i.e., 
its distributional derivatives belong to the Lorentz space
$L_{n,1}$, see section~\ref{prel} for the exact definitions) satisfies the~$N$-property. Note that this space  $W^1_{n,1}(\R^n,\R^n)$ is limiting in a natural sense between 
classes $W^1_n$ and $W^1_p$ with $p>n$. 

Another direction is to study the~$N$-properties with respect to Hausdorff (instead of Lebesgue) measures. One of the most elegant results was achieved for the class of 
plane quasiconformal mappings.

The famous area distortion theorem of K. Astala \cite{AstalaActa1994} implies the following dimension distortion result: if $f:\C\to\C$ is a $K$-quasiconformal mapping (i.e., it is a plane 
homeomorphic mapping with $K$-bounded distortion) and $E$ is a compact set of Hausdorff dimension $t\in(0,2)$, then the image $f(E)$ has Hausdorff dimension at most $t'=\frac{2Kt}{2+(K-1)t}$. This estimate is sharp; however, it leaves open the endpoint case: does $\H^t(E)=0$ imply $\H^{t'}(f(E))=0$? The remarkable paper~\cite{Lacey-Acta2-10}
gives an affirmative answer to Astala's conjecture (see also \cite{As_Riesz2013} where the further implication $H^t(E)<\infty\Rightarrow\H^{t'}(f(E))<\infty$ was considered). 

Let us go to the results which are more close to the~present paper. It is more
natural to discuss the topic in the scale of fractional Sobolev
spaces, i.e., for {\it(Bessel)-potential space} $\LK$ with
$\al>0$. Recall, that function $v:\R^n\to\R^d$ belongs to the
space $\LK$, if it is a convolution of the Bessel
kernel~$K_{\alpha}$ with a function~$g\in L_p(\R^n)$, where
$\widehat{K_\alpha}(\xi)=(1+4\pi^2\xi^2)^{-\alpha/2}$. It is well known that
\begin{equation*}\label{fN1} \LK(\R^n)=W^\al_p(\R^n)\qquad\mbox{ if }\ \al\in\N\quad\mbox{and}\quad 1<p<\infty.
\end{equation*}

In the recent paper~~\cite{H-H} \ Hencl H. and Honz\'{\i}k P.  proved, in particular, the following assertion: 

\begin{ttt}[see \cite{H-H}]
\label{HH1}{\sl Let $n,d\in\N$, $\alpha>0$, $p>1$, $\al p>n$, and $0<\tau\le n$. Suppose that a mapping~$f:\R^n\to\R^d$ belongs to the (fractional) Sobolev class 
$\LK$. Then for any set $E\subset\R^n$ with Hausdorff dimension~$\dim_HE\le \tau$ the inequality
$\dim_Hf(E)\le\sigma(\tau)$ holds, where 
\begin{equation}\label{np1}
\sigma(\tau):=\left\{\begin{array}{lcr}\tau, \ \ &\mbox{\rm if \ }& \tau\ge\tau_*:=n-(\alpha-1) p;\\ [13pt]
\frac{p\,\tau}{\alpha p-n+\tau}, \ \ &\mbox{\rm if \ }& 0<\tau <\tau_*.\end{array}\right.
\end{equation}
 }
\end{ttt}

But as above (see the discussion around the~Astala theorem), this result raises a natural question. What happens in the limiting case, 
i.e., is it true that $\H^\tau(E)=0$ implies $\H^{\sigma(\tau)}(f(E))=0$? 
Of course, such $N$-property is much more precise and stronger than the assertion of Theorem~\ref{HH1}.

 Six years ago G. Alberti et al. announced \cite{Alb2} (see also \cite{Alb}) the positive answer to this question, i.e., they announced the validity of the following theorem.

\begin{ttt}
\label{AK1}{\sl Let $k, n,d\in\N$, $p>1$, $k p>n$, and $0<\tau\le n$. Suppose that a mapping~$f:\R^n\to\R^d$ belongs to the  Sobolev class 
$W^k_p$ and $\tau\ne\tau_*=n-(k-1)p$. Then $f$ has the~$(\tau,\sigma)$-$N$-property, where the value $\sigma=\sigma(\tau)$ is defined in~(\ref{np1}).
 }
\end{ttt}

Here for convenience we use the following notation: a mapping~$f:\R^n\to\R^d$ is said to satisfy the $(\tau,\sigma)$-$N$-property, if $\H^\sigma(f(E))=0$ whenever $\H^\tau(E)=0$, \ $E\subset\R^n$. 

Let us note two things concerning this result. First of all, in announcements \cite{Alb2}--\cite{Alb} authors left the limiting case $\tau=\tau_*>0$ as {\it an open question}. 
Further, as far as we know, they did not publish a paper with the proofs of the~results announced in~\cite{Alb2}--\cite{Alb} (it was written in~\cite{Alb} that the work is still~"in progress"\,). 

In the present paper we extend the above assertion to the case of fractional Sobolev spaces and also we cover the critical case $\tau=\tau_*$ as well. 

\begin{ttt}
\label{LPT1}{\sl Let $\al>0$, \,$1<p<\infty$, \,$\alpha p>n$,  and
$v\in \LK(\R^n,\R^d)$. Suppose that $0<\tau\le n$. Then   the following assertions hold:
\begin{itemize}
\item[(i)] \,if $\tau\ne\tau_*=n-(\alpha-1)p$, then $v$ has
the~$(\tau,\sigma)$-$N$-property,  where the value $\sigma=\sigma(\tau)$ is defined in~(\ref{np1}).

\item[(ii)] \,if $\al>1$ \,and \,$\tau=\tau_*>0$ then $\sigma(\tau)=\tt$ and  the mapping $v$ in
general has NO \ $(\tt,\tt)$-$N$-property, i.e., it could be
$\H^{{\tt}}(v(E))>0$ for some $E\subset\R^n$ with
$\H^{{\tt}}(E)=0$.
\end{itemize}
}
\end{ttt}

\begin{rem}
\label{LPT1-rem}Note that if $\al=1$ and $p>n$, then $\tt=n$ and $\LK(\R^n,\R^d)=W^1_p(\R^n,\R^d)$, and the 
validity of $(\tau,\sigma)$-$N$-property for all $\tau\in(0,n]$ and for all 
mappings of these spaces is a~simple corollary of the classical Morrey inequality (see, e.g., \cite{MalyMartio}\,). 
\end{rem}

The above Theorem~\ref{LPT1} omits the limiting cases $\alpha p=n$
and $\tau=\tt$. It is possible to cover these cases as well using
the Lorentz norms. Namely, denote by $\LLL(\R^n,\R^d)$ the space
of functions which could be represented as a~convolution of
the~Bessel potential $K_\al$ with a function~$g$ from the Lorentz
space~$L_{p,1}$ (see the definition of these spaces in the
section~\ref{prel}); respectively,
$$\|v\|_{\LLL}:=\|g\|_{L_{p,1}}.$$

\begin{ttt}
\label{LPT2}{\sl Let $\al>0$, \,$1<p<\infty$, \,$\alpha p\ge n$,
and $0<\tau\le n$. Suppose that $v\in \LLL(\R^n,\R^d)$. Then $v$ is a~continuous function satisfying
the~$(\tau,\sigma)$-$N$-property,  where again the value $\sigma=\sigma(\tau)$ is defined in~(\ref{np1})
(i.e., the limiting case $\tau=\tt$ is INCLUDED).
}
\end{ttt}

\begin{rem}
\label{LPT2-rem}In the case $\al=k\in\N$, $kp=n$, $p\ge1$, we have $\tt=p$ and the validity of $(\tau,\sigma)$-$N$-property for mappings of the corresponding Sobolev--Lorentz 
space $W^k_{p,1}(\R^n,\R^d)$ was proved in the previous papers~\cite{BKK2}, \cite{KK15}.
\end{rem}

\subsection{The counterexample for the limiting case $\tau=\tt$ in Theorem~\ref{LPT1}-(ii).}

Suppose again that
\begin{equation*}
\label{ass3} n>(\al-1)p>n-p.
\end{equation*}
 Let us demonstrate that the positive assertion in Theorem~\ref{LPT1}~(i) \,is very sharp:
it fails in general for the limiting case~
\begin{equation*}
\label{aass3} \t=\tt=n-(\al-1)p.
\end{equation*}

Take
\begin{equation*}
\label{ass4} n=4,\qquad \al=2, \qquad p=3.
\end{equation*}
Then by definition
\begin{equation*}
\label{assa5} \tau_*=1.
\end{equation*}
So we have to construct a~function from the Sobolev space
$\Le^2_3(\R^4)=W^2_3(\R^4)$ which does not have $N$-property with
respect to $\H^1$-measure. Consider the restrictions (traces) of
functions from $W^4_3(\R^4)$ to the real line. It is well known
that the space of these traces coincides exactly with the Besov
space $B^1_{3,3}(\R)$ (see, e.g., \cite[Chapter~1, Theorem~4 on
p.~20]{JW}\,). Consider the function of one real  variable
$$f_\sigma(x)= e^{-x^2}
\sum\limits_{m=1}^\infty 5^{-m}\,{m^{-\sigma}}\,\cos(5^mx),$$
where
\begin{equation*}
\label{ass5} \frac13<\sigma<\frac12.
\end{equation*}
It is known that $f_\sigma\in B^1_{3,3}(\R)$ under above
assumptions (see, e.g., \S6.8 in Chapter~V of~\cite{St}\,). Nevertheless, the following result holds.

\begin{ttt}\label{cex1}{\sl The above function $f_\sigma:\R\to\R$ does not have $(1,1)$-$N$-property (with respect to $\H^1$-measure). }
\end{ttt}

This result follows directly from the following two classical
facts:

\begin{ttt}[see, e.g., Theorem~7.7 of Chapter~IX in \cite{Saks}]\label{cex2}{\sl If a function $f:\R\to\R$ has the $N$-property,
then it is differentiable on the set of positive measure. }
\end{ttt}

\begin{ttt}[see, e.g., \S6 of Chapter~V, page~206, in \cite{Zyg}]\label{cex3}{\sl The continuous function $f(x)=
\sum\limits_{m=1}^\infty b^{-m}\e_m\,\cos(b^mx)$ with $b>1$ and
$\e_m\to0$, $\sum\limits_{m=1}^\infty\e_m^2=\infty$, is not
differentiable almost everywhere. }
\end{ttt}

Note that the above functions $f_\sigma$, $f$ from
Theorems~\ref{cex1}, \ref{cex3} are the typical examples of the so
called lacunary Fourier series.

From Theorem~\ref{cex1} it follows that there exists a function
$v\in W^2_3(\R^4)$ such that its restriction to the real line
coincides with $f_\sigma$, i.e., $v$ does {\bf not} have the
$N$-property. The construction of the counterexample is finished.

\subsection{"The uncertainty principle"  and Fubini--type theorems for $N$-properties}

The above formulated $N$-properties have an important application in the recent extension of the Morse--Sard theorem to Sobolev spaces (\cite{FKR}, see also Subsection~\ref{MST}\,). 
Here we need the following notion. 

For a pair number $\tau,\sigma>0$ we will say that a continuous function
$v:\R^n\to\R^d$ satisfies the $(\tau,\sigma)$-$N_*$-property, if for every
$q\in[0,\sigma]$ and for any set $E\subset\R^n$ with $H^\tau(E)=0$
we have
\begin{equation}\label{dub4-q}
\H^{\tau(1-\frac{q}\sigma)}(E\cap v^{-1}(y))=0\qquad\mbox{ for \ $\H^q$-a.a.
}y\in\R^d.
\end{equation}
This implies, in particular, the usual $(\tau,\sigma)$-$N$-property
\begin{equation*}\label{dub4-q''}
\H^\sigma\bigl(v(E)\bigr)=0\qquad\mbox{ whenever}\quad\H^\tau(E)=0.
\end{equation*}
(Indeed, it is sufficient to take~$q=\sigma$ in~(\ref{dub4-q}).) In other words, $(\tau,\sigma)$-$N_*$-property is stronger than the~usual $(\t,\sigma)$ one. 

Intuitively, the sense of $N_*$-property is very close to
{\it Heisenberg's uncertainty principle} in theoretical physics:
the more precise information we received on measure of the image
of the critical set, the less precisely the preimages are
described, and vice versa.

Also, the $N_*$-property could be considered as Fubini type theorem for the~usual $N$-property.

Now we could strength our previous results in the following way. 

\begin{ttt}
\label{LPT1-m}{\sl Let $\al>0$, \,$1<p<\infty$, \,$\alpha p>n$,  and
$v\in \LK(\R^n,\R^d)$. Suppose that $0<\tau\le n$. Then   the following assertions hold:
\begin{itemize}
\item[(i)] \,if $\tau\ne\tau_*=n-(\alpha-1)p$, then $v$ has
the~$(\tau,\sigma)$-$N_*$-property,  where the value $\sigma=\sigma(\tau)$ is defined in~(\ref{np1}).

\item[(ii)] \,if $\al>1$ and $\tau=\tau_*$, then $\sigma(\tau)=\tt$ and  the mapping $v$ in
general has no \ $(\tt,\tt)$-$N$-property, i.e., it could be
$\H^{{\tt}}(v(E))>0$ for some $E\subset\R^n$ with
$\H^{{\tt}}(E)=0$.
\end{itemize}
}
\end{ttt}

\begin{rem}
\label{LPT1-rem*}Note that if $\al=1$ and $p>n$, then $\tt=n$ and $\LK(\R^n,\R^d)=W^1_p(\R^n,\R^d)$, and the 
validity of $(\tau,\sigma)$-$N_*$-property for all $\tau\in(0,n]$ and for all 
mappings of these spaces is a~simple corollary of the classical Morrey inequality and Theorem~\ref{ff} below. 
\end{rem}

Of course, the above Theorem~\ref{LPT1-m} omits the limiting cases $\alpha p=n$
and $\tau=\tt$. Again, it is possible to cover these cases as well using
the Lorentz norms. 

\begin{ttt}
\label{LPT2-m}{\sl Let $\al>0$, \,$1<p<\infty$, \,$\alpha p\ge n$,
and $0<\tau\le n$. Suppose that $v\in \LLL(\R^n,\R^d)$. Then $v$ is a~continuous function satisfying
the~$(\tau,\sigma)$-$N_*$-property,  where again the value $\sigma=\sigma(\tau)$ is defined in~(\ref{np1}).
}
\end{ttt}

\begin{rem}
\label{LPT2-rem*}In the case $\al=k\in\N$, $kp=n$, $p\ge1$, we have $\tt=p$ and the validity of $(\tau,\sigma)$-$N_*$-property for mappings of the corresponding Sobolev--Lorentz 
space $W^k_{p,1}(\R^n,\R^d)$ was proved in the previous papers~\cite{BKK2}, \cite{HKK}.
\end{rem}

\subsection{Application to the Morse--Sard--Dubovitski\u{\i}--Federer theorems}
\label{MST}

The classical Morse--Sard theorem claims  that for a mapping $v
\colon \R^n\to\R^m$ of class $\CC^k$ the measure of critical
values $v(Z_{v,m})$ is zero under condition $k
> \max(n-m, 0)$. Here the critical set, or
$m$-critical set is defined as $Z_{v,m} = \{ x \in \R^n : \, \rank
\nabla v(x) < m \}$. Further Dubovitski\u{\i}   in 1957~\cite{Du} and
independently Federer and Dubovitski\u{\i}  in 1967 (see \cite{Du2} and \cite[Theorem
3.4.3]{Fed}\,) found some
elegant extensions of this theorem to the case of other (e.g.,
lower) smoothness assumptions. They also established the sharpness
of their results within the $\CC^k$ category.

Recently (see \cite{FKR}) it was found the 
following \textit{bridge theorem} that
includes all the above results as particular cases.

We  say that a~mapping $v:\R^n\to\R^d$ belongs to the class $\Cc$
for some integer positive~$k$ and $0<\alpha\le1$ if there exists
a~constant $L\ge 0$ such that
\begin{equation*}\label{hic1}\mbox{$|\nabla^kv(x)-\nabla^kv(y)|\le
L\,|x-y|^\alpha$ \qquad for all $x,y\in\R^n$.}
\end{equation*}

To simplify the notation, let us make the following agreement: for
$\alpha=0$ we identify~$C^{k,\alpha}$ with usual spaces of 
$C^k$-smooth mappings. The following theorem was obtained in \cite{FKR}.

\begin{ttt}\label{DST-H}{\sl Let
$m\in\{1,\dots,n\}$, \,$k\ge1$, \,$d\ge m$, \,$0\le\al\le1$, \,and
\,$v\in \Cc(\R^n,\R^d)$. Then for any $q\in (m-1,\infty)$ the
equality
\begin{equation*}\label{hf1}
\H^{\mu_q}(Z_{v,m}\cap v^{-1}(y))=0\qquad\mbox{ for \ $\H^q$-a.a.
}y\in\R^d
\end{equation*}
holds, where
\begin{equation*}\label{hf2}\mu_q=n-m+1-(k+\alpha)(q-m+1),
\end{equation*} and
$Z_{v,m}$ denotes the set of $m$-critical points of~$v$:\ \
$Z_{v,m}=\{x\in\R^n\,:\,\rank\nabla v(x)\le m-1\}$.}
\end{ttt}

Here and in the following we interpret $\H^\beta$ as the counting
measure when $\beta\le0$. Let us note, that for the classical $C^k$-case, i.e., when
$\alpha=0$, the behavior of the function $\mu_q$ is very natural:
\begin{equation*}\label{dub3-q}
\begin{array}{lcr}
\mu_q=0\qquad\mbox{ for }q=\b=m-1+\frac{n-m+1}k\ \ \qquad\mbox{(Dubovitski\u{\i}--Federer Theorem~1967)};\\
[13pt]
\mu_q<0 \qquad\mbox{ for }q>\b\ \ \ \qquad\mbox{[ibid.]};\\
[13pt] \mu_q=n-m-k+1 \qquad\mbox{ for }q=m\ \ \qquad\mbox{(Dubovitski\u{\i} Theorem~1957)}; \\
[13pt]\mu_q=n-m+1\qquad\mbox{ for }q=m-1.\end{array}
\end{equation*}
The last value cannot be improved in view of the trivial example
of a linear mapping $L\colon \R^n\to\R^d$ of rank $m-1$.

Thus, Theorem~\ref{DST-H} contains all the previous theorems
(Morse--Sard, Dubovitski\u{\i}--Federer, and even the Bates theorem for
$C^{k,1}$-Lipschitz functions~\cite{Bates}) as particular cases.

Intuitively, the sense of this bridge theorem again is very close to
{\it Heisenberg's uncertainty principle} in theoretical physics:
the more precise information we received on measure of the image
of the critical set, the less precisely the preimages are
described, and vice versa.

The above bridge theorem was extended also in~\cite{FKR} to the case of fractional Sobolev spaces. 

Let $k\in\N$, $1<p<\infty$  and $0<\alpha<1$. There exist two main
types of fractional Sobolev spaces $W^{k+\alpha}_p(\R^n)$ (which
is a Sobolev analog of classical Holder classes $\Cc$), namely,
{\it (Bessel) potential spaces} $\Le^{k+\alpha}_p$ \ (see above) \ and {\it Besov
spaces} $B^{k+\al}_{p,s}$. We are not going to discuss many
elegant relations of these spaces. For our purposes it is
sufficient to indicate that
\begin{equation}\label{fs1}
\begin{array}{lcr}
\Le^{k+\alpha}_p\subset B^{k+\al}_{p,\infty};\\
[13pt] \forall s\in (1,\infty)\qquad B^{k+\al}_{p,s}\subset
B^{k+\al}_{p,\infty}.\end{array}
\end{equation}

So if one  proves the Bridge F.-D.-theorem for the largest space
$B^{k+\al}_{p,\infty}(\R^n,\R^d)$ (see below its definition~\ref{DST-df}\,), then automatically the result
will be true for any other $(k+\alpha,p)$-fractional Sobolev space
of above kind. 

In \cite{FKR} we obtained the following analog of the Bridge--Morse--Sard--Dubovitski\u{\i}--Federer Theorem for the Sobolev (items (i)-(ii)) and 
fractional Sobolev (items (iii)--(iv)) cases:

\begin{ttt}[\cite{FKR}]\label{DFT-F}{\sl Let
$m\in\{1,\dots,n\}$, \,$k\ge1$, \,$d\ge m$, \,$0\le\al<1$, \,$p\ge1$
\,and \,let $v:\R^n\to \R^d$ be a~mapping 
for which one of the following cases holds:

\begin{itemize}
\item[(i)] \,$\al=0$, \,$kp>n$, \,and \,$v\in W^k_p(\R^n,\R^d)$;

\item[(ii)] \,$\al=0$, \,$kp=n$, \,and \,$v\in W^k_{p,1}(\R^n,\R^d)$;

\item[(iii)] \,$0<\al<1$, \,$p>1$, \,$(k+\al)p>n$, \,and \,$v\in \BBB$;

\item[(iv)] \,$0<\al<1$, \,$p>1$, \,$(k+\al)p=n$, \,and \,$v\in \Le^{k+\al}_{p,1}(\R^n,\R^d)$.
\end{itemize}
Then for any $q\in (m-1,\infty)$ the equality
\begin{equation*}\label{sob-hf1-hhh}
\H^{\mu_q}(Z_{v,m}\cap v^{-1}(y))=0\qquad\mbox{ for \ $\H^q$-a.a.
}y\in\R^d
\end{equation*}
holds, where again 
\begin{equation*}\label{sob-hf2-hhh}\mu_q=n-m+1-(k+\alpha)(q-m+1),
\end{equation*} and
$Z_{v,m}$ denotes the set of $m$-critical points of~$v$:\ \
$Z_{v,m}=\{x\in\R^n\setminus A_v\,:\,\rank\nabla v(x)\le m-1\}$.
}
\end{ttt}

Here $A_v$ means the set of nondifferentiability points for~$v$.
Recall, that by approximation results (see, e.g.,
\cite{Sw} \,and \,\cite{KK15}\,) under conditions of Theorem~\ref{DFT-F} the equalities
\begin{eqnarray*}\label{hf0}\!\!\!\!\!\!\!\!\!\!\!\!\!\!\!\!\!\!\!\H^\tau(A_v)=0\qquad\forall\tau>\tt:=n-(k+\al-1)p\qquad\mbox{in cases (i), (iii)};\\
\label{hf0--}\!\!\!\!\!\!\!\!\H^{\tt}(A_v)=\H^p(A_v)=0\qquad\tt:=n-(k+\al-1)p=p\qquad\mbox{in cases (ii), (iv)}
\end{eqnarray*}
are valid (in particular, $A_v=\emptyset$ if $(k+\al-1)p>n$).
Our purpose is to prove that the impact of the "bad"\ set~$A_v$ is negligible
in the~ Bridge D.-F. Theorem~\ref{DFT-F}, i.e., that the following statement holds:

\begin{ttt}\label{DFT-F-negl}{\sl Let the conditions of Theorem~\ref{DFT-F} be fulfilled for a~function $v:\R^n\to\R^d$. Then
\begin{equation*}\label{dub7-qqqq} \H^{\mu_q}(A_v\cap
v^{-1}(y))=0\qquad\mbox{ for \ $\H^q$--a.a. }y\in\R^d
\end{equation*}
for any $q>m-1$.
}
\end{ttt}

\

\begin{rem}\label{remq0}Note, that since $\mu_q\le0$ \,for \,$q\ge\b=m-1+\frac{n-m+1}{k+\alpha}$, the assertions of 
Theorems~\ref{DFT-F}--\ref{DFT-F-negl} are equivalent to the equality
$0=\H^q\bigl[v(A_v\cup Z_{v,m})\bigr]$ \ for \ $q\ge\b$, so  
it is sufficient to check the assertions of 
Theorems~\ref{DFT-F}--\ref{DFT-F-negl} for
\,$q\in(m-1,\b]$ only.
\end{rem}

\

Finally, let us comment briefly that the merge ideas for the proofs 
are from our previous papers~\cite{BKK2}, \cite{KK3,KK15}
and~\cite{HKK}. In particular, the joint papers \cite{BKK,BKK2} by
one of the authors with J.~Bourgain contain many of the key ideas
that allow us to consider nondifferentiable Sobolev mappings. For
the implementation of these ideas one relies on estimates for the
Hardy--Littlewood maximal function in terms of Choquet type
integrals with respect to Hausdorff capacity. In order to take
full advantage of the Lorentz context we exploit the recent
estimates from~\cite{KK15} (recalled in Theorem~\ref{lb7} below,
see also ~\cite{Ad} for the case~$p=1$). 
\bigskip

\noindent {\em Acknowledgment.} M.K. was partially supported by
the Ministry of Education and Science of the Russian Federation
(Project number 1.8126.2017/8.9). The main part of the paper was written
during a visit of M.K. to the University of Campania "Luigi
Vanvitelli" in~2017, and he is very thankful for the hospitality.
Also the authors are very grateful to the~academician Sergei Konyagin and to the professor Sergey
Vodop'yanov for the fruitful discussion concerning the lacunary Fourier series and Besov
spaces.

\section{Preliminaries}
\label{prel}

\noindent

\noindent By  an {\it $n$--dimensional interval} we mean a closed
cube in $\R^n$ with sides parallel to the coordinate axes. If $Q$
is an $n$--dimensional cubic interval then we write $\ell(Q)$ for
its sidelength.

For a subset $S$ of $\R^n$ we write $\Le^n(S)$ for its outer
Lebesgue measure (sometimes we use the symbol~$\meas S$ for the
same object\,). The $m$--dimensional Hausdorff measure is denoted
by $\H^m$ and the $m$--dimensional Hausdorff content by
$\H^{m}_{\infty}$. Recall that for any subset $S$ of $\R^n$ we
have by definition
$$
\H^m (S)=\lim\limits_{t\searrow 0}\H^m_t (S) = \sup_{t >0}
\H^{m}_{t}(S),
$$
where for each $0< t \leq \infty$,
$$
\H^m_t (S)=\inf\left\{ \sum_{i=1}^\infty(\diam S_i)^m\ :\ \diam
S_i\le t,\ \ S \subset\bigcup\limits_{i=1}^\infty S_i \right\}.
$$
It is well known that $\H^n(S)  = \H^n_\infty(S)\sim\Le^n(S)$ for
sets~$S\subset\R^n$ \ (\,"$\sim$"\ \  means, here and in the following, that these values have upper and lower bounds with positive constants independent on the set~$E$\,).

By $L_p(\R^n)$,  $1\le p \le \infty$, we will denote the usual Lebesgue space equipped with the norm $\|\cdot \|_{L_p}$. The notation  $\|f\|_{L_{p}(E)}$ means $\|1_E\cdot f\|_{L_{p}}$,
where $1_E$ is the indicator function of~$E$.

Working with locally integrable functions, we always assume that
the precise representatives are chosen. If $w\in
L_{1,\loc}(\Omega)$, then the precise representative $w^*$ is
defined for {\em all} $x \in \Omega$ by
\begin{equation*}
\label{lrule}w^*(x)=\left\{\begin{array}{rcl} {\displaystyle
\lim\limits_{r\searrow 0} \dashint_{B(x,r)}{w}(z)\,\dd z}, &
\mbox{ if the limit exists
and is finite,}\\
 0 \qquad\qquad\quad & \; \mbox{ otherwise},
\end{array}\right.
\end{equation*}
where the dashed integral as usual denotes the integral mean,
$$
\dashint_{B(x,r)}{ w}(z) \, \dd
z=\frac{1}{\Le^n(B(x,r))}\int_{B(x,r)}{ w}(z)\,\dd z,
$$
and $B(x,r)=\{y: |y-x|<r\}$ is the open ball of radius $r$
centered at $x$. Henceforth we omit special notation for the
precise representative writing simply $w^* = w$.

\noindent For $0\le \beta <n$, the {\it fractional maximal function} of $w$ of order $\beta$ is given by
\begin{equation}\label{maximal}
M_\beta w\,(x)=\sup_{r>0}\, r^\beta \dashint_{B(x,r)}|w(z)|dz\,.
\end{equation}
When $\beta=0$, $M_0$ reduce to the usual Hardy-Littlewood maximal operator $M$. 
\medskip

The Sobolev space $\WW^{k}_{p} (\R^n,\R^d)$ is as usual defined as
consisting of those $\R^d$-valued functions $f\in \LL_p(\R^n)$
whose distributional partial derivatives of orders $l\le k$ belong
to $\LL_p(\R^n)$ (for detailed definitions and differentiability
properties of such functions see, e.g., \cite{EG}, \cite{M},
\cite{Ziem}, \cite{Dor}). Denote by $\nabla^k f$ the vector-valued
function consisting of all $k$-th order partial derivatives of $f$
arranged in some fixed order. However, for the case of first order
derivatives $k=1$ we shall often think of $\nabla f(x)$ as the
Jacobi matrix of $f$ at $x$, thus the $d \times n$ matrix whose
$r$-th row is the vector of partial derivatives of the $r$-th
coordinate function.

We use the norm
$$
\|f\|_{\WW^{k}_{p}}=\|f\|_{\LL_p}+\|\nabla
f\|_{\LL_p}+\dots+\|\nabla^kf\|_{\LL_p},
$$
 and unless otherwise specified all norms on the
spaces $\R^s$ ($s \in \N$) will be the usual euclidean norms.

If $k<n$, then it is well-known that functions from Sobolev spaces
$\WW^{k}_{p}(\R^n)$ are continuous for $p>\frac{n}k$ and could be
discontinuous for $p\le \p=\frac{n}k$ (see, e.g., \cite{M,Ziem}).
The Sobolev--Lorentz space $\WW^{k}_{\p,1}(\R^n)\subset
\WW^{k}_{\p}(\R^n)$ is a refinement of the corresponding Sobolev
space. Among other things functions that are locally in
$\WW^{k}_{\p,1}$ on $\R^n$ are in particular continuous.

Here we only mentioned the Lorentz space $\LL_{p,1}$, and in this
case one may rewrite the norm as (see for instance
\cite[Proposition 3.6]{Maly2})
\begin{equation*}\label{lor1}
\|f\|_{L_{p,1}}=
\int\limits_0^{+\infty}\bigl[\Le^n(\{x\in\R^n:|f(x)|>t\})\bigr]^{
\frac1p} \, \dd t.
\end{equation*}
As for Lebesgue norm we set $\|f\|_{L_{p,1}(E)}:=\|1_E\cdot f\|_{L_{p,1}}$.
Of course, we have the inequality
\begin{equation}\label{lor1---}
\|f\|_{L_p}\le \|f\|_{L_{p,1}}.\end{equation}
Moreover, recall, that by properties of Lorentz spaces, the standard
estimate \begin{equation} \label{max-1} \|M f\|_{L_{p,q}}\le
C\,\|f\|_{L_{p,q}}
\end{equation}
holds for $1<p<\infty$ (see, e.g., \cite[Theorem 4.4]{Maly2}\,).

Denote by $\WW^{k}_{p,1}(\R^n)$ the space of all functions $v\in
\WW^k_p(\R^n)$ such that in addition the Lorentz norm~$\|\nabla^k
v\|_{\LL_{p,1}}$ is finite.

\subsection{On the~largest Besov spaces~$B^{k+\al}_{p,\infty}(\R^n,\R^d)$}

Recall the following definition which was used in the Bridge--Morse--Sard Theorem~\ref{DFT-F}.

\begin{df}\label{DST-df}{\sl Let $k\in\N$, $1<p<\infty$,  and $0<\alpha<1$. 
We will say that a~mapping  $v:\R^n\to\R^d$ belongs to the class $B^{k+\al}_{p,\infty}(\R^n,\R^d)$, \ if \
$v\in W^k_p(\R^n,\R^d)$ and there exists a~constant $C$ such that
for any $t>0$ the estimate
\begin{equation*}\label{fs2}
\|\Omega^k_v(\cdot,t)\|_{L^p(\R^n)}\le C\,t^{\alpha}
\end{equation*}
holds,  where
\begin{equation*}\label{fs3}
\Omega^k_v(x,t)=\dashint_{Q(x,t)}|\nabla^kv(x)-\nabla^kv(Q)|\,dx,
\end{equation*}
and $\nabla^kv(Q)$ denotes the mean value of \  $\nabla^kv$ over
the~$n$-dimensional cube~$Q=Q(x,t)$ centered at~$x$ of
size~$t=\ell(Q)$.}
\end{df}

This is the largest space among  other $(k+\alpha,p)$-fractional Sobolev space $B^{k+\al}_{p,q}$ and $\Le^{k+\al}_p$ (see~(\ref{fs1})\,).

\subsection{On potential spaces~$\LK$}

In the paper we deal with {\it(Bessel)-potential space} $\LK$ with
$\al>0$. Recall that a~function $v:\R^n\to\R^d$ belongs to the
space $\LK$, if it is a convolution of the Bessel
kernel~$K_{\alpha}$ with a function~$g\in L_p(\R^n)$:
$$v=\GG_\alpha(g):=K_\al*g,$$
where
$\widehat{K_\alpha}(\xi)=(1+4\pi^2\xi^2)^{-\alpha/2}$. In
particular, \
$$\|v\|_{\LK}:=\|g\|_{L_p}.$$
It is well known that
\begin{equation}\label{fN1----} \LK(\R^n)=W^\al_p(\R^n)\qquad\mbox{ if }\ \al\in\N\quad\mbox{and}\quad 1<p<\infty\,,
\end{equation}
and we use the agreement that 
$\LK(\R^n)=L_p(\R^n)$ when $\al=0$. Moreover, the following well-known result holds:

\begin{ttt}[see, e.g., Lemma~3 on page 136 in~\cite{St}]\label{propert-pot1}{\sl
Let $\alpha\ge1$ and $1<p<\infty$. Then $v\in
\LK(\R^n)$ iff $v\in\Le^{\al-1}_p(\R^n)$ \ and \ $\frac{\partial v}{\partial x_j}\in \Le^{\al-1}_p(\R^n)$ for every $j=1,\dots,n$. }
\end{ttt}

The following technical bounds will be used on several
occasions (for the convenience, we prove them in the Appendix\,).

\begin{lem}\label{est-potent1}{\sl
Let $\alpha>1$, \,$\alpha p>n$, \,and \,$p>1$. Suppose that $v\in
\LK(\R^n)$, i.e., \ $v=\GG_\al(g)$ for some $g\in L_p(\R^n)$. Then for every $n$-dimensional cubic
interval $Q\subset \R^n$ with $r=\ell(Q)\le 1$ the estimate
\begin{equation}
\label{est-pot-1} \diam v(Q)\le
C\,\biggl[\|Mg\|_{L_p(Q)}r^{\alpha-\frac{n}p}+\frac1{r^{n-1}}\int\limits_QI_{\al-1}|g|(y)\,dy\,\biggr]
\end{equation}
holds, where the constant~$C$
depends on $n,p,d,\alpha$ only, \ and
$$I_{\beta}f(x):=\int_{\R^n}\frac{f(y)}{|y-x|^{n-
{\beta}}}\,\dd y$$ is  the Riesz potential of order $\beta$. }
\end{lem}

Sometimes it is not convenient to work with the~Riesz potential, and we need also the following variant of above estimates. 

\begin{lem}\label{est-potent2}{\sl
Let $\alpha>0$, \,$\alpha p>n$, \,and \,$p>1$. Suppose that $v\in
\LK(\R^n)$, i.e., \ $v=\GG_\al(g)$ for some $g\in L_p(\R^n)$.  Fix arbitrary  $\theta>0$ such that 
$\al+\theta\ge1$. Then for every $n$-dimensional cubic
interval $Q\subset \R^n$ with $r=\ell(Q)\le 1$ the estimate
\begin{equation}
\label{est-pot-2} \diam v(Q)\le
C\,\biggl[\|Mg\|_{L_p(Q)}r^{\alpha-\frac{n}p}+\frac1{r^{n+\theta-1}}\int\limits_QM_{\al-1+\theta}g(y)\,dy\,\biggr]
\end{equation}
holds, where the constant~$C$
depends on $n,p,d,\alpha,\theta$ only. }
\end{lem}

For reader's convenience, we prove Lemmas~\ref{est-potent1}--\ref{est-potent2} in the
Appendix~\ref{appendix}.

\subsection{On Lorentz potential spaces~$\LLL$}

To cover some other limiting cases,
denote by $\LLL(\R^n,\R^d)$ the space
of functions which could be represented as a~convolution of
the~Bessel potential $K_\al$ with a function~$g$ from the Lorentz
space~$L_{p,1}$; respectively,
$$\|v\|_{\LLL}:=\|g\|_{L_{p,1}}.$$

Because of inequality~(\ref{lor1---}), we have the~evident inclusion
\begin{equation*}\label{lor2}
\LLL(\R^n)\subset \LK(\R^n).\end{equation*}

Since these spaces are not so common, let us discuss briefly some of their properties. We need some technical facts concerning the Lorentz spaces. 

\begin{lem}[see, e.g., \cite{Georgia1} ] \label{lorl-1}{\sl
Let $1<p<\infty$. Then for any $j=1,\dots,n$ the Riesz transform $\mathscr R_j$ is continuous from $L_{p,1}(\R^n)$ \ to \ $L_{p,1}(\R^n)$. 
}
\end{lem}

\begin{lem}[see, e.g., \cite{Schep1}] \label{lorl-11}{\sl
Let $1<p<\infty$ and $\mu$ be a finite Borel measure on~$\R^n$. Then the convolution transform $f\mapsto f*\mu$ is continuous in the space
$L_{p,1}(\R^n)$ \ and  in \ $\LK(\R^n)$ for all~$\al>0$.
}
\end{lem}

Using these facts and repeating almost "word by word" the arguments from \cite[\S3.3 and 3.4]{St}, one could obtain the following very natural results.

\begin{ttt}[cf. with Lemma~3 on page 136 in~\cite{St}]\label{propert-pot-lor1}{\sl
Let $\alpha\ge1$ and $1<p<\infty$. Then $f\in
\LLL(\R^n)$ iff $f\in\Le^{\al-1}_{p,1}(\R^n)$ \ and \ $\frac{\partial f}{\partial x_j}\in \Le^{\al-1}_{p,1}(\R^n)$ for every $j=1,\dots,n$. }
\end{ttt}

\begin{cor}\label{propert-pot-lor2}{\sl
Let $k\in\N$ and $1<p<\infty$. Then $\Le^k_{p,1}(\R^n)=W^k_{p,1}(\R^n)$, where 
$W^k_{p,1}(\R^n)$ is the space of functions such that all its distributional partial derivatives of order~$\le k$ belong to~$L_{p,1}(\R^n)$.  }
\end{cor}
Note, that the space $W^k_{p,1}(\R^n)$ admits even more simple (but equivalent) description: it consists of functions $f$ from the usual Sobolev space~$W^k_p(\R^n)$  satisfying the additional condition 
$\nabla^kf\in L_{p,1}(\R^n)$ (i.e., this condition is on the highest derivatives only), see, e.g.,~\cite{Maly2}. 

As before, we need some standard estimates. 

\begin{lem}\label{est-crit-lor}{\sl
Let $\alpha>0$, \,$\alpha p\ge n$, \,and \,$p>1$. Suppose that $v\in
\LLL(\R^n)$, i.e., \ $v=\GG_\al(g)$ for some $g\in L_{p,1}(\R^n)$.  
Then the
function~$v$ is continuous and for every $n$-dimensional cubic
interval $Q\subset \R^n$ with $\ell(Q)\le 1$ the estimate
\begin{equation}
\label{est-pot-4} \diam v(Q)\le
C\,\biggl[\|Mg\|_{L_{p,1}(Q)}r^{\alpha-\frac{n}p}+\frac1{r^{n+\theta-1}}\int\limits_QM_{\al-1+\theta}g(y)\,dy\,\biggr]
\end{equation}
holds for arbitrary (fixed) parameter $\theta>0$ such that 
$\al+\theta\ge1$ (here  the constant~$C$
again depends on $n,p,d,\alpha,\theta$ only).  Furthermore, if $\al>1$, then 
\begin{equation}
\label{est-pot-3} \diam v(Q)\le
C\,\biggl[\|Mg\|_{L_{p,1}(Q)}r^{\alpha-\frac{n}p}+\frac1{r^{n-1}}\int\limits_QI_{\al-1}|g|(y)\,dy\,\biggr].
\end{equation}
}
\end{lem}

For reader's convenience, we prove Lemma~\ref{est-crit-lor} in the
Appendix~\ref{appendix}.

\subsection{On Choquet type integrals}

Let $\Me^\beta$ be the space of all nonnegative Borel
measures~$\mu$ on~$\R^n$ such that
\begin{equation*}
\label{mu1} | \! | \! |\mu | \! | \!
|_{\beta}=\sup_{I\subset\R^n}\ell(I)^{- \beta}\mu(I)<\infty,
\end{equation*}
where the supremum is taken over all $n$--dimensional  cubic
intervals $I\subset\R^n$ and $\ell(I)$ denotes side--length
of~$I$.

Recall the following classical theorem referred to
D.R.~Adams.

\bigskip

\begin{ttt}[see, e.g., \S1.4.1 in~\cite{M} or \cite{Ad1} ]\label{AM1}{\sl
Let $\alpha>0$, \,$n-\alpha p>0$, \,$s>p>1$ \,and \,$\mu$ be a
positive Borel measure on $\mathbb R^n$. Then for any $g\in
L_p(\R^n)$ the estimate
\begin{equation}
\label{Ri1} \int \bigl| I_\al g\bigr|^s\,\dd\mu \leq C| \! | \! |\mu |
\! | \! |_\beta\cdot\Vert g\Vert^s_{L_p}
\end{equation}
holds with $\beta=\frac{s}p(n-\al p)$, where $C$ depends on $n,
\ p, \ s, \  \al$ only. }
\end{ttt}
\bigskip

The above estimate (\ref{Ri1}) fails for the limiting case $s=p$.
Namely, there exist functions $g\in L_p(\R^n)$ such that
$| I_\al g|(x)=+\infty$ on some set of positive
$(n-\al p)$--Hausdorff measure. Nevertheless, 
there are two ways to cover this limiting case~$s=p$. First way is using the maximal function $M_\al$ instead of Riesz potential in the left hand side of~(\ref{Ri1}). 

\begin{ttt}[see, e.g., Theorem~7 on page 28 in~\cite{Ad3} ]\label{AM2}{\sl
Let $\alpha>0$, \,$n-\alpha p>0$, \,$s\ge p>1$ \,and \,$\mu$ be a
positive Borel measure on $\mathbb R^n$. Then for any $g\in
L_p(\R^n)$ the estimate
\begin{equation}
\label{Ri2} \int \bigl| M_\al g\bigr|^s\,\dd\mu \leq C| \! | \! |\mu |
\! | \! |_\beta\cdot\Vert g\Vert^s_{L_p}
\end{equation}
holds with $\beta=\frac{s}p(n-\al p)$, where $C$ depends on $n,
\ p, \ s, \  \al$ only. }
\end{ttt}
\bigskip

The second way is using the Lorentz norm instead of Riesz potential in the right hand side of~(\ref{Ri1}).  Such possibility was proved in the recent paper~\cite{KK15}.

\begin{ttt}[see Theorem~0.2 in~\cite{KK15}]\label{lb7}{\sl
Let $\alpha>0$, \,$n-\alpha p>0$, \,and \,$\mu$ be a
positive Borel measure on $\mathbb R^n$. Then for any $g\in
L_p(\R^n)$ the estimate
\begin{equation*}
\label{Ri5} \int \bigl| I_\al g\bigr|^p\,\dd\mu \leq C| \! | \! |\mu |
\! | \! |_\beta\cdot\Vert g\Vert^p_{L_{p,1}}
\end{equation*}
holds with $\beta=n-\al p$, where $C$ depends on $n,
\ p, \  \al$ only. }
\end{ttt}

\subsection{On Fubini type theorems  for $N$-properties}

Recall that by usual Fubini theorem, if a set $E\subset\R^2$ has 
zero plane measure, then for $\H^1$-almost all straight lines $L$ parallel to coordinate axes we have $\H^1(L\cap E)=0$. Next 
result could be considered as Fubini type theorem for $N$-property. 

\begin{ttt}[see Theorem~5.3 in~\cite{HKK}]\label{FubN}{\sl
Let $\mu\ge 0$, \,$q>0$, \,and \,$v:\R^n\to\R^d$ \,be a continuous function. For a set $E\subset\R^n$
define the set function
\begin{equation*}\label{dd5}
\Phi(E)=\inf\limits_{E\subset\bigcup_j
D_j}\sum\limits_j\bigl(\diam D_j\bigr)^\mu\bigl[\diam
v(D_j)\bigr]^q,
\end{equation*}
where the infimum is taken over all countable families of compact
sets $\{D_j\}_{j\in \N}$ such that $E\subset\bigcup_j D_j$. Then \ $\Phi(\cdot)$ is a countably
subadditive and the implication
\begin{equation*}\label{dd6} \Phi(E)=0\ \boldsymbol{\Rightarrow}\
\biggl[\H^\mu\bigl(E\cap v^{-1}(y)\bigr)=0\quad\mbox{for
$\H^q$-almost all }y\in\R^d\biggr]
\end{equation*}holds.}
\end{ttt}

\subsection{On local properties of considered potential spaces}\label{loc-s1}
Let $\Be$ be some space of functions defined on~$\R^n$. For a set $\Omega\subset\R^n$ define the space $\Be_\loc(\Omega)$ in the following standard way:
$$\Be_\loc(\Omega):=\{f:\Omega\to\R:\mbox{ for any compact set }E\subset\Omega\mbox{ $\exists g\in\Be$ such that }f(x)=g(x)\ \forall x\in E\,\}.$$
Put for simplicity $\Be_\loc=\Be_\loc(\R^n)$.

Is is easy to see that for $\al>0$ and $q>p>1$ the following inclusions hold:
\begin{equation*}\label{loc1} \Le^\al_{q,\loc}\subset \Le^\al_{p,1,\loc}\subset \Le^\al_{p,\loc}.
\end{equation*}
Since the $N$-properties have a local nature, this means that if we prove some $N$- (or $N_*$\,) properties for $\Le^\al_{p}$, then the same $N$-property will be valid for 
the spaces $\Le^\al_{p,1}$ and $\Le^\al_{q}$ for all $q>p$. Similarly, if we prove some $N$- (or $N_*$\,) properties for $\Le^\al_{p,1}$, then the same $N$-property will be valid for 
the spaces $\Le^\al_{q}$ with $q>p$, etc.

\section{Proofs of $N$-properties (\,Theorems~\ref{LPT1}--~\ref{LPT2}\,) }

\bigskip
In this Section we will prove Theorems \ref{LPT1}--\ref{LPT2}. For each Theorem, we will consider  different cases. The most interesting case is when $\alpha \,p <n+p$, which implies that $\tau_*>0$: in such situation we will consider the supercritical case $\tau>\tau_*>0$ and the undercritical case $0<\tau<\tau_*$ (see, respectively, Sections \ref{t1-super-a} and \ref{t1-under} below). The case $\alpha \,p \ge n+p$ is contained in Section \ref{t1-super-b}.
\medskip

\noindent In the proofs we will consider particular family of intervals to cover a given set, whose properties are more  suitable for our aims.  
Below {\bf a dyadic interval} means  a closed cube
in $\R^n$ of the form $[\frac{k_1}{2^l},\frac{k_1+1}{2^l}]\times\dots\times[\frac{k_n}{2^l},\frac{k_n+1}{2^l}]$,
where $k_i,l$ are integers. Denote
$$
\Lambda^s(E)=\inf\left\{ \sum_{i=1}^\infty \ell(Q_i)^s\ :\  E\subset\bigcup\limits_{i=1}^\infty Q_i \ \ Q_i \ \ \mbox{dyadic } \right\}.
$$
It is well known that $\Lambda^s (E)\sim \H^s(E)$  for all subset $E\subset\R^n$;  \ in particular,  $\Lambda^s $ and $\H^s$ have the same null sets. 

\noindent Let  $\{ Q_j \}_{j \in \N}$ be a family of $n$-dimensional dyadic
intervals. For a~given parameter~$\tau>0$ we say that the family $\{ Q_j \}$ is {\bf regular}, if $\sum\ell(Q_j)^{\tau}<\infty$ and 
for any $n$-dimensional dyadic interval $Q$ the estimate
\begin{equation}\label{q8}
\ell(Q)^{{\tau}}\ge\sum\limits_{j : Q_j\subset
Q}\ell(Q_j)^{{{\tau}}}
\end{equation}
holds. Since dyadic intervals are either nonoverlapping or
contained in one another, (\ref{q8}) implies that any regular
family $\{ Q_j \}$ must in particular consist of nonoverlapping
intervals. Moreover, the following result holds.

\begin{lem}[see Lemma~2.3 in \cite{BKK2}]\label{lb5.1}{\sl
Let $\{J_i \}$ be a family of $n$--dimensional dyadic intervals with $\sum\limits_{i}\ell(
J_i)^{{{\tau}}}<\infty$.
Then there exists
 a regular family $\{Q_j \}$ of $n$--dimensional
dyadic intervals such that $\bigcup_i J_i\subset \bigcup_j Q_j$
and
$$
\sum\limits_{j}\ell(Q_j)^{{{\tau}}}\le\sum\limits_{i}\ell(
J_i)^{{{\tau}}}.
$$
}\end{lem}

\subsection{Proof of the Theorem~\ref{LPT1}: the supercritical case~$\tau>\tau_*>0$.}\label{t1-super-a}

Fix the parameters $n\in\N$, $\alpha>0$, $p>1$ such that 
\begin{equation}\label{n-dd7}
\al p>n,\ \qquad\tt=n-(\al-1)p>0,\end{equation}
and take \begin{equation}\label{assc1}
\t\in(\tt,n].\end{equation}
Fix also a mapping $v\in \LK(\R^n,\R^d)$.  If $\al=1$, then $v\in W^1_p(\R^n)$ with $p>n$ and $\tau=n$, and the result is well-known. 
So we restrict our attention to the nontrivial case~$\alpha>1$, $\tau<n$.

Now let $\{ Q_i \}_{i\in\N}$ be a regular family of $n$--dimensional
dyadic intervals. Take any family of points $x_i\in Q_i$ and
consider the corresponding measure $\mu$ defined as
\begin{equation}
\label{mmu} \int
f\,d\mu:=\sum\limits_i\frac{1}{\ell(Q_i)^{n-\tau}}\int\limits_{Q_i}f(y)\,dy.
\end{equation}
As usual, for a measurable set $E\subset\R^n$ put $\mu(E)=\int
1_E\,d\mu$, where $1_E$ is an indicator function of~$E$.

\begin{lem}[see, e.g., Lemma~2.4
in \cite{KK3}]\label{regm}{\sl For any
 regular family $\{ Q_i \}_{i\in\N}$ of $n$--dimensional
dyadic intervals the
corresponding measure~$\mu$ defined by~(\ref{mmu})  satisfies
\begin{equation*}
\label{mm1} \mu(Q)\le \ell(Q)^\tau,
\end{equation*}
for any dyadic cube $Q\subset\R^n$. }\end{lem}

From this fact and from the  Adams theorem~\ref{AM1}  we have immediately

\begin{lem}\label{Thh3.3}{\sl Let $g\in L_p(\R^n)$. Then 
for each $\varepsilon>0$ there exists
$\delta=\delta(\varepsilon,v)>0$ such that for any regular family
$\{ Q_i \}$ of $n$--dimensional dyadic intervals the estimate
\begin{equation}
\label{oxe1} \sum\limits_i\frac1{\ell(Q_i)^{n-\tau}}\int\limits_{Q_i}\bigl(I_{\alpha-1}|g|\bigr)^s\,dy\le C\|g\|_{L_p}^s
\end{equation}holds, where
$s:=\frac{\tau }{\tau_*}p>p$ \,and \,$C$ does not depend on~$g$.}\end{lem}

\

Now we are ready to formulate the key step of the proof.

\begin{lem}\label{Thh3.3-77}{\sl Under above assumptions, 
for each $\varepsilon>0$ there exists
$\delta=\delta(\varepsilon,v)>0$ such that for any regular family
$\{ Q_i\}$ of $n$--dimensional dyadic intervals if 
\begin{equation*}
\label{cas1} 
\sum\limits_{i}\ell(Q_i)^{\tau}<\delta,
\end{equation*}
then 
\begin{equation*}
\label{oxe4} \sum\limits_i\bigl[\diam v(Q_i)\bigr]^\tau<\e.
\end{equation*}}
\end{lem}

\begin{proof}
Since $v \in \LK(\R^n , \R^d )$, by definition of this space, it is
easy to see that for any $\epsilon>0$ there exists a
representation
$$v=v_1+v_2,$$
where $v_i\in \LK(\R^n , \R^d )$, \,$v_1\in C^\infty(\R^n)$,
\begin{equation*}
\label{mi11}\|\nabla v_1\|_{L_\infty(\R^n)}<\infty,
\end{equation*}
and
\begin{equation}
\label{mi12}v_2=\GG_\alpha(g) \qquad\mbox{ with }\ \|g\|_{L_p}<\ep.
\end{equation}
It means, in particular, that
\begin{equation}
\label{mi13}|\nabla v_1(x)|<K\qquad\forall x\in\R^n,
\end{equation}for some $K=K(\epsilon,v)\in\R$.
Take any regular family $\{ Q_i \}$ of $n$-dimensional dyadic
intervals such that
\begin{equation}
\label{mi14} \sum_i\ell(Q_i)^{{\tau}}<\delta
\end{equation}
(the exact value of~$\delta$ will be specified below). Put $r_i=\ell(Q_i)$. Then by Lemma~\ref{est-potent1}
\begin{equation*}
\label{mi15} \sum\limits_i\bigl[\diam v(Q_i)\bigr]^\tau\le C(S_1+S_2+S_3),\end{equation*} where
\begin{equation*}
\label{mi16} S_1=\sum\limits_i\bigl[\diam v_1(Q_i)\bigr]^\tau\overset{\footnotesize{(\ref{mi13})-(\ref{mi14})}}\le
K^\tau\,\delta,
\end{equation*}
and
\begin{equation*}
\label{mi17} S_2=\sum\limits_i\|Mg\|^\tau_{L_p(Q_i)}r_i^{\tau(\alpha-\frac{n}p)},
\end{equation*}
\begin{equation*}
\label{cmi17} S_3=\sum\limits_i\biggl(\frac1{r_i^{n-1}}\int\limits_{Q_i}I_{\al-1}|g|(y)\,dy\biggr)^\tau.
\end{equation*}
Let us estimate~$S_2$. Since by assumptions~(\ref{n-dd7}) the inequality $\al-\frac{n}p<1$ holds, we could apply the Holder inequality to obtain
\begin{equation*}
S_2\le\biggl(\sum\limits_i\|Mg\|^{\tau\frac{p}{n-p(\alpha-1)}}_{L_p(Q_i)}\biggr)^{\frac{n}p-\alpha+1}\cdot\biggl(\sum\limits_i r^\tau_i\biggr)^{\al-\frac{n}p}
\overset{\footnotesize{(\ref{mi14})}}\le \biggl(\sum\limits_i\|Mg\|^{\tau\frac{p}{n-p(\alpha-1)}}_{L_p(Q_i)}\biggr)^{\frac{n}p-\alpha+1}\cdot\delta^{\al-\frac{n}p}=
\end{equation*}
\begin{equation*}\label{calc1}
\overset{\footnotesize{(\ref{n-dd7})}}=\biggl(\sum\limits_i\|Mg\|^{p\frac\tau\tt}_{L_p(Q_i)}\biggr)^{\frac\tt{p}}\cdot\delta^{\al-\frac{n}p}\overset{\footnotesize{(\ref{assc1})}}\le
\|Mg\|^{\tau}_{L_p(\cup_iQ_i)}\cdot\delta^{\al-\frac{n}p}\overset{\footnotesize{(\ref{mi12})}}\le\ep^\tau\cdot\delta^{\al-\frac{n}p}
\end{equation*}
Similarly, taking $s=\frac\tau\tt p$ and applying twice the Holder inequality  in~$S_3$ \ (first time~--- for the integrals, and the second time~--- for sums), we obtain
\begin{equation*}
S_3\le\sum\limits_i\biggl(\int\limits_{Q_i}\bigl(I_{\al-1}|g|\bigr)^s\,dy\biggr)^{\frac\tt{p}}\cdot r_i^{n(\tau-\frac\tt{p})}
\cdot r_i^{(1-n)\t}
=\sum\limits_i\biggl(\frac1{r_i^{n-\t}}\int\limits_{Q_i}\bigl(I_{\al-1}|g|\bigr)^s\,dy\biggr)^{\frac\tt{p}}\cdot r_i^{(1-\frac\tt{p})\tau}
\end{equation*}
\begin{equation*}\label{calc2}
\overset{\footnotesize{\rm\color{red}Holder}}\le\biggl(\sum\limits_i\frac1{r_i^{n-\t}}\int\limits_{Q_i}\bigl(I_{\al-1}|g|\bigr)^s\,dy\biggr)^{\frac\tt{p}}\cdot\biggr(\sum\limits_i r_i^{\tau}\biggr)^{1-\frac\tt{p}}
\overset{\footnotesize{(\ref{oxe1}),\,(\ref{mi12}),\,(\ref{mi14}) }}=
\ C\,\ep^\tau\cdot\delta^{1-\frac\tt{p}}.
\end{equation*}
So taking $\delta$ sufficiently small such that $K^\tau\,\delta<\frac12\e$ is small, we have
that $S_1+S_2+S_3<\e$ as required. The lemma~\ref{Thh3.3-77} is proved.
\end{proof}

Finally, if  $E$ is a set such that $\H^\tau(E)=0$, then also $\Lambda^\tau (E)=0$, and this lemma together with lemma~\ref{lb5.1} implies 
the validity of the assertion Theorem~\ref{LPT1}~(i) for the supercritical case~$\tau>\tt>0$. 

\subsection{Proof of the Theorem~\ref{LPT1}: the undercritical case~$0<\tau<\tau_*$.}\label{t1-under}
Now
fix the parameters $n\in\N$, $\alpha>0$, $p>1$ such that 
\begin{equation}\label{und-n-dd7}
\al p>n,\ \qquad\ \ \tt=n-(\al-1)p>0,\end{equation}
and take \begin{equation*}\label{und-assc1}\t\in(0,\tt),\qquad \sigma=
\frac{p\,\tau}{\alpha p-n+\tau}.\end{equation*}
Evidently, by this definition
\begin{equation}\label{und-assc7}\sigma>\tau.\end{equation} 
Fix also a mapping $v\in \LK(\R^n,\R^d)$.  
Take an additional parameter $\theta$ such that 
\begin{equation*}
\label{und-5} (\al-1+\theta)>0\qquad\mbox{ and }\qquad n-(\al-1+\theta)p>0.
\end{equation*}

From Lemma~\ref{regm}
and from the  Adams theorem~\ref{AM2} (\,taking $s=p$\,) we have immediately

\begin{lem}\label{und-Thh3.3}{\sl Let $g\in L_p(\R^n)$. Then 
for each $\varepsilon>0$ there exists
$\delta=\delta(\varepsilon,v)>0$ such that for any $\tau$-regular family
$\{ Q_i \}$ of $n$--dimensional dyadic intervals the estimate
\begin{equation}
\label{und-oxe1} \sum\limits_i\frac1{\ell(Q_i)^{n-\t_\th}}\int\limits_{Q_i}\bigl(M_{\alpha-1+\theta}|g|\bigr)^p\,dy\le C\|g\|_{L_p}^p
\end{equation}holds, where $\t_\th=n-(\alpha-1+\th)p$
and \,$C$ does not depend on~$g$.}\end{lem}

\noindent As in the previous case, the proof of Theorem \ref{LPT1} in the case $0<\tau<\tau^*$ will be complete, once we establish the following result.

\begin{lem}\label{lem-under}{\sl Under above assumptions, 
for each $\varepsilon>0$ there exists
$\delta=\delta(\varepsilon,v)>0$ such that for any regular family
$\{ Q_i\}$ of $n$--dimensional dyadic intervals if 
\begin{equation*}
\label{und-cas1} 
\sum\limits_{i}\ell(Q_i)^{\tau}<\delta,
\end{equation*}
then 
\begin{equation*}
\label{und-oxe4} \sum\limits_j\bigl[\diam v(Q_i)\bigr]^\sigma<\e.
\end{equation*}}
\end{lem}

\begin{proof}

Again, since function $v \in \LK(\R^n , \R^d )$, by definition of this space, for any $\epsilon>0$ there exists a
representation
$$v=v_1+v_2,$$
where $v_i\in \LK(\R^n , \R^d )$, \,$v_1\in C^\infty(\R^n)$,
\begin{equation*}
\label{und-mi11}\|\nabla v_1\|_{L_\infty(\R^n)}<\infty,
\end{equation*}
and
\begin{equation}
\label{und-mi12}v_2=\GG_\alpha(g) \qquad\mbox{ with }\ \|g\|_{L_p}<\ep.
\end{equation}
It means, in particular, that
\begin{equation}
\label{und-mi13}|\nabla v_1(x)|<K\qquad\forall x\in\R^n,
\end{equation}for some $K=K(\epsilon,v)\in\R$.
Take any regular family $\{ Q_i \}$ of $n$-dimensional dyadic
intervals such that
\begin{equation}
\label{und-mi14} \sum_i\ell(Q_i)^{{\tau}}<\delta<1
\end{equation}
(the exact value of~$\delta$ will be specified below). Put $r_i=\ell(Q_i)$. Then by Lemma~\ref{est-potent2}
\begin{equation*}
\label{und-mi15} \sum\limits_i\bigl[\diam v(Q_i)\bigr]^\sigma\le C(S_1+S_2+S_3),\end{equation*} where
\begin{equation*}
\label{und-mi16} S_1=\sum\limits_i\bigl[\diam v_1(Q_i)\bigr]^\sigma\overset{\footnotesize{(\ref{und-assc7}),\,(\ref{und-mi13})-(\ref{und-mi14})}}\le
K^\sigma\,\delta,
\end{equation*}
and
\begin{equation*}
\label{und-mi17} S_2=\sum\limits_i\|Mg\|^\sigma_{L_p(Q_i)}r_i^{\sigma(\alpha-\frac{n}p)},
\end{equation*}
\begin{equation*}
\label{und-cmi17} S_3=\sum\limits_i\biggl(\frac1{r_i^{n-1+\theta}}\int\limits_{Q_i}M_{\al-1+\th}\,g(y)\,dy\biggr)^\sigma.
\end{equation*}
Let us estimate~$S_2$. Since by assumptions~(\ref{und-n-dd7}) the inequality $\sigma<p$ holds and 
\begin{equation}
\label{uund-cmi17}\frac{p-\sigma}p=\frac{\al p-n}{\al p-n+\tau},\ \qquad\ \sigma\,\frac{p}{p-\sigma}=\frac\tau{\al-\frac{n}p}\end{equation}
we could apply the Holder inequality to obtain
\begin{equation*}
S_2\le\biggl(\sum\limits_i\|Mg\|^p_{L_p(Q_i)}\biggr)^{\frac{\sigma}p}\cdot\biggl(\sum\limits_i r_i^{\sigma (\al -\frac{n}p)\frac{p}{p-\sigma}}\biggr)^{\frac{p-\sigma}p}=
\biggl(\|Mg\|^p_{L_p(\cup_iQ_i)}\biggr)^{\frac{\sigma}p}\cdot\biggl(\sum\limits_i r_i^{\tau}\biggr)^{\frac{p-\sigma}p}
\overset{\footnotesize{(\ref{und-mi14}), (\ref{und-mi12})}}\le \ep^{\sigma}\delta^{1-\frac\sigma{p}}.
\end{equation*}
Similarly, applying twice the Holder inequality  in~$S_3$ \ (first time~--- for the integrals, and the second time~--- for sums), we obtain
\begin{equation*}
S_3\le\sum\limits_i\biggl(\int\limits_{Q_i}\bigl(M_{\al-1+\th}g\bigr)^p\,dy\biggr)^{\frac{\sigma}{p}}\cdot r_i^{n\frac{p-1}p\sigma}
\cdot r_i^{(1-n-\th)\sigma}
=\sum\limits_i\biggl(\frac1{r_i^{n-\t_\th}}\int\limits_{Q_i}\bigl(M_{\al-1+\th}|g|\bigr)^p\,dy\biggr)^{\frac\sigma{p}}\cdot r_i^{(\al-\frac{n}{p})\sigma}
\end{equation*}
\begin{equation*}
\overset{\footnotesize{\rm\color{red}Holder}}\le\biggl(\sum\limits_i\frac1{r_i^{n-\t_\th}}\int\limits_{Q_i}\bigl(M_{\al-1+\th}|g|\bigr)^p\,dy\biggr)^{\frac\sigma{p}}\cdot\biggr(\sum\limits_i r_i^{(\al-\frac{n}{p})\sigma\frac{p}{p-\sigma}}\biggr)^{1-\frac\sigma{p}}\end{equation*}
\begin{equation*}\overset{\footnotesize{(\ref{uund-cmi17}) }}=\biggl(\sum\limits_i\frac1{r_i^{n-\t_\th}}\int\limits_{Q_i}\bigl(M_{\al-1+\th}|g|\bigr)^p\,dy\biggr)^{\frac\sigma{p}}\cdot\biggr(\sum\limits_i r_i^{\tau}\biggr)^{1-\frac\sigma{p}}
\overset{\footnotesize{(\ref{und-oxe1}),\,(\ref{und-mi12}),\,(\ref{und-mi14}) }}=
\ C\,\ep^\sigma\cdot\delta^{1-\frac\sigma{p}}.\end{equation*}
So taking $\delta$ sufficiently small such that $K^\tau\,\delta<\frac12\e$ is small, we have
that $S_1+S_2+S_3<\e$ as required. The Lemma is proved.
\end{proof}

Finally, we conclude exactly as in the previous case.

\subsection{Proof of the Theorem~\ref{LPT1}: the supercritical case~$\tau_*\le0<\tau$.}\label{t1-super-b}
Consider now the case  $\al p>n$ and $\tau_*=n-(\al-1)p\le0$. If $(\al-1)p> n$, then  every function $v\in \LK(\R^n,\R^d)$
is locally Lipschitz (even $C^1$\,) and the result is trivial.
Suppose now $(\al-1)p=n$.  Under these assumptions, let $\tau>0$
and $v\in \LK(\R^n,\R^d)$. Take a number $1<\tilde p<p$ \  such
that $\alpha\tilde p>n$ and $\tau>\tt=n-(\alpha-1)\tilde p>0$.
Then we have that $v\in \Le^\al_{\tilde p,\loc}(\R^n,\R^d)$ (see the~subsection~\ref{loc-s1}\,).
Therefore, by previous case $\tau>n-(\tilde\al-1)p>0$ the mapping~$v$ has the
$(\tau,\tau)$-$N$-property. $\qed$

\subsection{Proof of the Theorem~\ref{LPT2}.}

The proof of Theorem \ref{LPT2} is very similar to that one of Theorem \ref{LPT1}: the main differences concern the limiting cases $\alpha p =n$ or $\tau=\tau^*$.
\begin{itemize}
\item{\sc Case $\alpha p >n$ and $\tau \neq \tau^*$.} The required assertion follows immediately from Theorem~\ref{LPT1} and from the
inclusion $\LLL(\R^n)\subset\LK(\R^n)$ \ (this inclusion follows
from the definitions of these space and from the
relation~$L_{p,1}(\R^n)\subset L_p(\R^n)$\,). 

\item{\sc Case $\alpha p =n$ and $\tau>\tau_*>0$.} The required assertion can be proved repeating almost 
"word by word" the same arguments as the supercritical case in the previous Theorem~\ref{LPT1} with 
the following evident modifications: now one has to apply the~estimate~(\ref{est-pot-3}) (which covers the case $\al p=n$\,) 
instead of previous estimate~(\ref{est-pot-1}), and, in addition, one needs the following analog of the additivity property for the Lorentz norms:
$$\sum_i\|f\|^p_{\LL_{p,1}(Q_i)}\le
\|f\|^p_{\LL_{p,1}(\cup_iQ_i)},$$ for any family of disjoint cubes (see, e.g.,
\cite[Lemma~3.10]{Maly2}).

\item{\sc Case $\alpha p \ge n$ and $\tau=\tau^*$.} The required assertion can be proved repeating almost "word by word" the same arguments as the supercritical case in the previous Theorem~\ref{LPT1} with 
the following evident modifications: now $\tau=\tau_*$ (this simplifies a little bit the~calculations\,) and one has to apply Theorem~\ref{lb7} (which covers the case~$s=p$\,) 
and the estimate~(\ref{est-pot-3})
instead of previous Theorem~\ref{AM1} (where $s>p$\,) and the inequality~(\ref{est-pot-1}), respectively. 

\item{\sc Case $\alpha p =n$ and $0<\tau< \tau^*$.} By a direct calculation,  we get $\sigma(\tau)\equiv p$ for any $\tau\in (0,\tau_*]$, and the result 
follows from the above considered critical case $\tau=\tau_*$.

\end{itemize}

\medskip

Thus both Theorems~\ref{LPT1}--\ref{LPT2} are proved completely.

\begin{rem}\label{summ}Really, we have proved that under assumptions of Theorems~\ref{LPT1}--\ref{LPT2}, for every fixed function~$v:\R^n\to\R^d$ from the considered potential spaces 
and for the corresponding pair~$(\tau,\sigma)$ the following assertion holds: 
for any $\e>0$ there exists $\delta>0$ such that for every $\tau$-regular family of cubes $Q_i\subset\R^n$ \,if \,$\sum\limits_i\ell(Q_i)^\tau<\delta$, \ then \ 
$\sum\limits_i\bigl[\diam v(Q_i)\,\bigr]^\sigma<\e.$
\end{rem}

\section{Proof of "Fubini type"  \ $N_*$-properties }

\bigskip

Here we have to prove  Theorems~\ref{LPT1-m}--\ref{LPT2-m}. We need the following general fact.

\begin{ttt}\label{ff}{\sl Let $\tau\in(0,n]$,\ $\sigma>0$,  and let $v:\R^n\to\R^d$ be~a continuous function. 
 Suppose that for any $E\subset \R^n$ \,with \,$\H^\tau(E)=0$ \,and for every $\e>0$ there exists a family of
compact
sets $\{D_i\}_{i\in \N}$ such that 
\begin{equation}\label{ff-0}
\mbox{$E\subset\bigcup_i D_i$ \qquad and \qquad $\sum\limits_i[\diam D_i\,\bigr]^\tau<\e$ \qquad and \qquad $\sum\limits_i\bigl[\diam v(D_i)\,\bigr]^\sigma<\e$.}\end{equation}
 Then 
$v$ has the $(\tau,\sigma)$--$N_*$--property, i.e.,  for every
$q\in[0,\sigma]$ and for any set $E\subset\R^n$ with $H^\tau(E)=0$
we have
\begin{equation}\label{ff-1}
\H^{\tau(1-\frac{q}\sigma)}(E\cap v^{-1}(y))=0\qquad\mbox{ for \ $\H^q$-a.a.
}y\in\R^d.
\end{equation}}
\end{ttt}

\medskip

{\bf Proof}.
Let the assumptions of the Theorem be fulfilled. Fix $q\in[0,\sigma]$. If $q=0$ or $q=\sigma$, then the required assertion~(\ref{ff-1}) follows trivially from these assumptions. 
Suppose now that 
\begin{equation*}\label{ff-2}
0<q<\sigma.
\end{equation*}
Fix arbitrary $\e>0$ and take the corresponding sequence of compact sets~$D_i$ satisfying~(\ref{ff-0}). Put $\mu=\tau(1-\frac{q}\sigma)<\tau$.
Then
$$\sum\limits_j\bigl(\diam D_i\bigr)^\mu\bigl[\diam
v(D_i)\bigr]^q\overset{\footnotesize{\rm\color{red}Holder}}\le \biggl(\sum\limits_i\bigl[\diam D_i\bigr]^{\mu\frac{\sigma}{\sigma-q}}\biggr)^{1-\frac{q}\sigma}\cdot\biggl(\sum\limits_i\bigl[\diam v(D_i)\bigr]^\sigma\biggr)^{\frac{q}{\sigma}}
$$
$$=
\biggl(\sum\limits_i\bigl[\diam D_i\bigr]^{\tau}\biggr)^{1-\frac{q}\sigma}
 \biggl(\sum\limits_i\bigl[\diam v(D_i)\bigr]^\sigma\biggr)^{\frac{q}{\sigma}}\overset{\footnotesize{(\ref{ff-0}) }}<\e.
$$
Since $\e>0$ was arbitrary, now the required assertion follows immediately from the Theorem~\ref{FubN}. $\qed$

\ 

The obtained Theorem~\ref{ff} and the Remark~\ref{summ} imply evidently the assertions of Theorems~\ref{LPT1-m}--\ref{LPT2-m}.

\subsection{Proof of the Theorem~\ref{DFT-F-negl}.}

Fix a mapping $v:\R^n\to\R^d$ for which the assumptions of Theorem~\ref{DFT-F} are fulfilled. We have to prove that 
\begin{equation}\label{---dub7-qqqq} \H^{\mu_q}(A_v\cap
v^{-1}(y))=0\qquad\mbox{ for \ $\H^q$--a.a. }y\in\R^d
\end{equation}
for any $q>m-1$, where $\mu_q=n-m+1-(k+\al)(q-m+1)$ and $A_v$ is the set of nondifferentiability points of~$v$.
Recall that, by approximation results (see, e.g.,
\cite{Sw} \,and \,\cite{KK15}\,) under conditions of Theorem~\ref{DFT-F} the equalities
\begin{eqnarray}\label{dimq1}\!\!\!\!\!\!\!\!\!\!\!\!\!\!\!\!\!\!\!\H^\tau(A_v)=0\qquad\forall\tau>\tt:=n-(k+\al-1)p\qquad\mbox{in cases (i), (iii)};\\
\label{dimq2}\!\!\!\!\!\!\!\!\H^{\tt}(A_v)=\H^p(A_v)=0\qquad\tt:=n-(k+\al-1)p=p\qquad\mbox{in cases (ii), (iv)}
\end{eqnarray}
are valid.

Because of Remark~\ref{remq0} we could assume without loss of generality that 
\,$q\in(m-1,\b]$.
Then for all cases (i)--(iv) we have
\begin{equation*}\label{negl0}\biggl(\frac{n}{k+\alpha}\le p\biggr)\Rightarrow\biggl( q-m+1\le\b-m+1=\frac{n-m+1}{k+\al}\le p \biggr)\Rightarrow\end{equation*}
$$\Rightarrow\mu_q=n-m+1-(k+\al)(q-m+1)=n-(k+\al-1)\,(q-m+1)-q\ge n-(k+\al-1)p-q=\tt-q.$$ In other words,
\begin{equation}\label{negl1} \mu_q\ge \tt-q,
\end{equation}
where the equality holds iff
\begin{equation}\label{negl2}k=1,\ \ \al=0,\ \ \mu_q=n-q=\tt-q\end{equation} 
or \begin{equation}\label{negl3}m=1,\ \ (k+\al)p=n,\ \ 
q=p=\tt, \ \ \mu_q=0.
\end{equation} 

Below for convenience  we consider the cases Theorem~\ref{DFT-F}-(i)--(iv) separately. 

\

{\sc Case I ($\al=0$, \ $kp>n,\ \ p\ge 1,\ \ v\in W^k_p(\R^n,\R^d)$\,).} This case splits into the following three subcases. 

\

{\sc Case Ia ($k=1,\ \ p>n,\ \ \tt=n,\ \ \mu_q=n-q$\,).} Then the required assertion~(\ref{---dub7-qqqq}) follows immediately from the equality $\H^n(A_v)=0$ and from the Remark~\ref{LPT1-rem*}.

\

{\sc Case Ib ($\tt<0$ or $\tt=0,k=n+1,p=1$\,).} Then the set $A_v$ is empty (since functions of the space $W^k_p(\R^n,\R^d)$ are $C^1$-smooth), and there is nothing to prove. 

\

{\sc Case Ic ($\tt\ge0,\ \  p>1,\ \ k>1,\ \ kp>n$\,).}  Then by (\ref{dimq1}) we have
\begin{equation}\label{negl7}
\forall\tau>\tt\qquad 
\H^\tau(A_v)=0.
\end{equation}
Further, by Theorem~\ref{LPT1-m} function~$v$ has $(\tau,\tau)$-$N_*$-property for every $\tau>\tt$. This implies, in particular, by virtue of~(\ref{negl7}), 
that for every $\tau>\tt$ and for every $q\in[0,\tau]$ the equalities
\begin{equation}\label{negl6}
\H^{\tau-q}(A_v\cap
v^{-1}(y))=0\qquad\mbox{ for \ $\H^q$--a.a. }y\in\R^d
\end{equation}
hold. 
Fix $q\in(m-1,\b]$ and take $\tau=q+\mu_q$. Since by construction $\mu_q\ge0$, we have $\tau\ge q$.  
Moreover,  by (\ref{negl1})--(\ref{negl3}) we have $\tau>\tt$. 
The last two inequalities together with~(\ref{negl6}) imply
\begin{equation*}\label{negl8}
\H^{\mu_q}(A_v\cap
v^{-1}(y))=0\qquad\mbox{ for \ $\H^q$--a.a. }y\in\R^d.
\end{equation*}
So the required assertion is proved for this case.

\ 

{\sc Case II ($\al=0$, \ $kp=n,\ \ p\ge 1,\ \ v\in W^k_{p,1}(\R^n,\R^d)$\,).}  In this case by definitions
\begin{equation*}\label{negl10}
\tt:=n-(k-1)p=p,
\end{equation*}
and, by (\ref{dimq2}) we have
\begin{equation}\label{negl11}
\H^p(A_v)=0.
\end{equation}
Further, by \cite[Theorem~2.3]{HKK} function~$v$ has $(\tau,\tau)$-$N_*$-property for every $\tau\ge p$. This implies, in particular, by virtue of~(\ref{negl11}), 
that for every $\tau\ge p$ and for every $q\in[0,\tau]$ the equalities
\begin{equation}\label{negl12}
\H^{\tau-q}(A_v\cap
v^{-1}(y))=0\qquad\mbox{ for \ $\H^q$--a.a. }y\in\R^d
\end{equation}
hold. 
Fix $q\in(m-1,\b]$ and take $\tau=q+\mu_q$. Since by construction $\mu_q\ge0$, we have $\tau\ge q$.  
Moreover,  by (\ref{negl1})--(\ref{negl3}) we have $\tau\ge\tt=p$. 
The last two inequalities together with~(\ref{negl12}) imply
\begin{equation*}\label{negl13}
\H^{\mu_q}(A_v\cap
v^{-1}(y))=0\qquad\mbox{ for \ $\H^q$--a.a. }y\in\R^d.
\end{equation*}
So the required assertion is proved for this case.

\ 

{\sc Case III ($0<\al<1$, \ $(k+\al)p>n,\ \ p> 1,\ \ v\in B^{k+\al}_{p,\infty}(\R^n,\R^d)$\,).}  If $\tt=n-(k+\al-1)p<0$, then $A_v=\emptyset$ and there is nothing to prove. Suppose now that $\tt\ge0$.
Since $B^{k+\al}_{p,\infty}(\R^n)\subset\Le^{k+\alpha-\e}_p(\R^n)$
for any $\e>0$, we obtain  from Theorem~\ref{LPT1-m} that 
$v$ has
the~$(\tau,\tau)$-$N_*$-property for every $\tau>{\tt}:=n-(\al-1)p$.
This implies, in particular, by virtue of~(\ref{dimq2}), 
that for every $\tau>\tt$ and for every $q\in[0,\tau]$ the equalities
\begin{equation}\label{negl14}
\H^{\tau-q}(A_v\cap
v^{-1}(y))=0\qquad\mbox{ for \ $\H^q$--a.a. }y\in\R^d
\end{equation}
hold. 
Fix $q\in(m-1,\b]$ and take $\tau=q+\mu_q$. Since by construction $\mu_q\ge0$, we have $\tau\ge q$.  
Moreover,  by (\ref{negl1})--(\ref{negl3}) we have $\tau>\tt$. 
The last two inequalities together with~(\ref{negl14}) imply
\begin{equation*}\label{negl15}
\H^{\mu_q}(A_v\cap
v^{-1}(y))=0\qquad\mbox{ for \ $\H^q$--a.a. }y\in\R^d.
\end{equation*}
So the required assertion is proved for this case.

\ 

{\sc Case IV ($0<\al<1$, \ $(k+\al)p=n,\ \ p> 1,\ \ v\in \Le^{k+\al}_{p,1}(\R^n,\R^d)$\,).}  
In this case by definitions
\begin{equation*}\label{negl17}
\tt:=n-(k-1)p=p,
\end{equation*}
and, by (\ref{dimq2}) we have
\begin{equation}\label{negl18}
\H^p(A_v)=0.
\end{equation}
Further, by Theorem~\ref{LPT2-m} function~$v$ has $(\tau,\tau)$-$N_*$-property for every $\tau\ge p$. This implies, in particular, by virtue of~(\ref{negl18}), 
that for every $\tau\ge p$ and for every $q\in[0,\tau]$ the equalities
\begin{equation}\label{negl19}
\H^{\tau-q}(A_v\cap
v^{-1}(y))=0\qquad\mbox{ for \ $\H^q$--a.a. }y\in\R^d
\end{equation}
hold. 
Fix $q\in(m-1,\b]$ and take $\tau=q+\mu_q$. Since by construction $\mu_q\ge0$, we have $\tau\ge q$.  
Moreover,  by (\ref{negl1})--(\ref{negl3}) we have $\tau\ge\tt=p$. 
The last two inequalities together with~(\ref{negl19}) imply
\begin{equation*}\label{negl20}
\H^{\mu_q}(A_v\cap
v^{-1}(y))=0\qquad\mbox{ for \ $\H^q$--a.a. }y\in\R^d.
\end{equation*}
So the required assertion is proved for this case, which is the last one.

Thus Theorem~\ref{DFT-F} is proved completely. $\qed$

\section{Appendix}\label{appendix}. 

Here we would like to prove the technical estimates of Lemmas~\ref{est-potent1}--\ref{est-potent2}
and  \ref{est-crit-lor}. Since now,  fix $\al>0$ and a cube $Q\subset\R^n$ of size~$r=\ell(Q)\le1$. Recall that by $2Q$ we denote the double cube with the same centre as~$Q$ of size $\ell(2Q)=2\ell(Q)$. 
We need some general elementary estimates. 

\begin{lem}\label{lem-a1} {\sl For any measurable function $g:\R^n\to\R_+$ and for every $x\in Q$ the inequality
\begin{equation}\label{a-est1}
\int\limits_{2Q}\frac{g(y)}{|x-y|^{n-\al}}\,dy\le \int\limits_{Q}\frac{Mg(y)}{|x-y|^{n-\al}}\,dy
\end{equation}
holds.}
\end{lem}

\begin{proof}
Fix $x\in Q$. Denote $r_0=\frac72\sqrt{n}r$. In particular, $2Q\subset B(x,\frac12r_0)$.

Now put $r_j=2^{-j}r_0$ \ and \ $B_j=B(x,r_j)\setminus B(x,r_{j+1})$, \ $j\in\N$. Evidently,
\begin{equation}\label{a-est0}
2Q=\bigcup\limits_{j\in\N}\bigl(2Q\cap B_j\bigr)\end{equation}
and 
\begin{equation}\label{a-est2}
\meas(Q\cap B_j)\ge C\, r_j^n \qquad \ \forall j\in\N
\end{equation}
(here and henceforth we denote by~$C$ the general constants
depending on the parameters~$n,p,d,\alpha$ only\,). 

Since $|x-y|\sim r_j$ for $y\in B_j$, by definition of the maximal function,  it is easy to see that the estimate 
$$\int\limits_{2Q\cap B_j}\frac{g(y)}{|x-y|^{n-\al}}\,dy\le   C\,r_j^{\al}\,Mg(z)\ \ \qquad\forall z\in Q\cap B_j$$
holds for all $j\in\N$. Integrating this inequality  with respect to~$z\in Q\cap B_j$ and using~(\ref{a-est2}), we have
\begin{equation}\label{a-est3}
\int\limits_{2Q\cap B_j}\frac{g(y)}{|x-y|^{n-\al}}\,dy\le   C\,r_j^{\al-n}\int\limits_{Q\cap B_j}Mg(z)\,dz.
\end{equation}
Since $|x-z|\sim r_j$ for $z\in Q\cap B_j$, the last inequality implies
\begin{equation}\label{a-est4}
\int\limits_{2Q\cap B_j}\frac{g(y)}{|x-y|^{n-\al}}\,dy\le   C\,\int\limits_{Q\cap B_j}\frac{Mg(y)}{|x-y|^{n-\al}}\,dy.
\end{equation}
Then summing these inequalities for all $j\in\N$ and taking into account~(\ref{a-est0}), we obtain the required estimate~(\ref{a-est1}).
\end{proof}

Since $K_\al(r)\le C r^{n-\al}$, see, e.g., \cite{AdH}, from the above lemma we have immediately 

\begin{cor}\label{lem-a2} {\sl For any measurable function $g:\R^n\to\R_+$ and for every $x\in Q$ the estimate 
\begin{equation}\label{a-est8}
\int\limits_{2Q}{g(y)}\,K_\al(x-y)\,dy\le C\,\int\limits_{Q}\frac{Mg(y)}{|x-y|^{n-\al}}\,dy
\end{equation}
holds.}
\end{cor}

We need also 

\begin{lem}\label{lem-a3} {\sl Let $v(x)=\GG_\al(x):=\int\limits_{\R^n}g(y)\,K_\al(x-y)\,dy,$ where $K_\al$ is
the corresponding Bessel potential function. Suppose $g\in L_p(\R^n)$ for some $1<p<\infty$ and 
\begin{equation}\label{a-est5}
g(y)\equiv 0\qquad\forall y\in 2Q.
\end{equation}
Then for arbitrary positive parameter $\theta\ge 1-\al$ the estimate
\begin{equation} \label{a-est6}
\diam \bigl[v(Q)\,\bigr]\le C\,r^{1-\theta-n}\int\limits_{Q}M_{\al+\theta-1}g(y)\,dy
\end{equation}
holds.
}
\end{lem}

\begin{proof}
Let the assumptions of the lemma be fulfilled. 
Without loss of generality suppose that $Q$ is centred at the origin. Since 
\begin{equation} \label{a-est9}
C_1|y|\le |y-x|\le C_2|y|\qquad\ \ \forall x\in Q,\ \ \forall y\in\R^n\setminus 2Q,
\end{equation}
and 
\begin{equation} \label{a-est10}
K'_\al(\rho)\le C \rho^{n-\al-1}
\end{equation} (see, e.g., \cite{AdH}\,) 
it is easy to deduce that 
\begin{equation} \label{a-est11}
\diam \bigl[v(Q)\,\bigr]\le\sup\limits_{x_1,x_2\in Q}\int\limits_{\R^n\setminus 2Q}|g(y)|\,\bigl[K_\al(x_1-y)-K_\al(x_2-y)\bigr]\,dy\le 
C\,r\, \int\limits_{\R^n\setminus 2Q}\frac{|g(y)|}{|y|^{n-\al+1}}\,dy.
\end{equation} 
Fix $\theta>0$ such that 
\begin{equation} \label{a-est12}
\al+\theta-1\ge0.\end{equation} 
Put $r_0=\frac12 r$, $r_j=2^jr_0$,  and consider a sequence of sets $B_j=B(0,r_{j+1})\setminus B(0,r_j)$. By construction, 
\begin{equation}\label{a-est13}
\R^n\setminus{2Q}\subset \bigcup\limits_{j\in\N}B_j.\end{equation}
and 
\begin{equation}\label{a-est15}
\int\limits_{B_j}\frac{|g(y)|}{|y|^{n-\al+1}}\,dy\le C\, r_j^{-\theta}r_j^{\al+\theta-1}\dashint\limits_{B_j}|g(y)|\,dy\le C\,r_j^{-\theta}M_{\al+\theta-1}g(0). 
\end{equation}
Therefore, using the elementary geometrical progression formula, we obtain
\begin{equation}\label{a-est16}
\int\limits_{\R^n\setminus{2Q}}\frac{|g(y)|}{|y|^{n-\al+1}}\,dy\le C\, M_{\al+\theta-1}g(0)\sum\limits_{j=1}^\infty r_j^{-\theta}\le
C\,r^{-\theta} \,M_{\al+\theta-1}g(0).
\end{equation}
It is easy to check (using the assumption that $g\equiv 0$ on $2Q$\,) that $M_{\al+\theta-1}g(0)\le C\,M_{\al+\theta-1}g(z)$ for every $z\in Q$.  Therefore,
\begin{equation}\label{a-est17}
 M_{\al+\theta-1}g(0)\le C\,\dashint\limits_{Q}M_{\al+\theta-1}g(z)\,dz,
\end{equation}
thus 
\begin{equation}\label{a-est18}
\int\limits_{\R^n\setminus{2Q}}\frac{|g(y)|}{|y|^{n-\al+1}}\,dy\le 
C\,r^{-\theta-n} \,\int\limits_{Q}M_{\al+\theta-1}g(z)\,dz. 
\end{equation}
Finally we obtain from~(\ref{a-est11}) that 
\begin{equation} \label{a-est19}
\diam \bigl[v(Q)\,\bigr]\le\
C\,r^{1-\theta-n} \,\int\limits_{Q}M_{\al+\theta-1}g(z)\,dz \end{equation} 
as required.
\end{proof}

Using the same arguments, with some evident simplifications, we could establish also the following estimate using the Riesz potentials:

\begin{lem}\label{lem-a4} {\sl Let $v(x)=\GG_\al(x):=\int\limits_{\R^n}g(y)\,K_\al(x-y)\,dy,$ where $K_\al$ is
the corresponding Bessel potential function. Suppose that $\al>1$,\ \ $g\in L_p(\R^n)$ for some $1<p<\infty$, and 
\begin{equation}\label{a-est21}
g(y)\equiv 0\qquad\forall y\in 2Q.
\end{equation}
Then the estimate
\begin{equation} \label{a-est20}
\diam \bigl[v(Q)\,\bigr]\le C\,r^{1-n}\int\limits_{Q}I_{\al-1}|g|(y)\,dy
\end{equation}
holds, where, recall, 
$$
I_{\al-1}|g|(x):=\int_{\R^n} \frac{|g|(y)}{|x-y|^{n-\al}}\,\dd y,
$$
 is the corresponding Riesz potential of the function~$|g|$.
}
\end{lem}

Using the established lemmas, it is very easy to finish the proof of required Lemmas~\ref{est-potent1}--\ref{est-potent2}
and  \ref{est-crit-lor}.  Indeed, 
fix $\al>0$, a cube $Q\subset\R^n$ of size~$r=\ell(Q)\le1$, and a function 
$v(x)=\GG_\al(x)=\int\limits_{\R^n}g(y)\,K_\al(x-y)\,dy$ \ with some  \ $g\in L_p(\R^n)$, \ $1<p<\infty.$

Split our
function~$v$ into the sum
\begin{equation} \label{max-ap2---} v=v_1+v_2,
\end{equation}
where
$$v_1:=\int\limits_{\R^n}g_1(y)\,K_\al(x-y)\,dy,\qquad \
v_2:=\int\limits_{\R^n}g_2(y)\,K_\al(x-y)\,dy,$$ and
$$g_1:=g\cdot 1_{2Q},\qquad\ g_2:=g\cdot 1_{\R^n\setminus 2Q}.$$

Suppose now that 
\begin{equation} \label{a-est24}
\al p>n\qquad\mbox{ and }\qquad\al>1
\end{equation}
Then from Corollary~\ref{lem-a2}, applying the Holder inequality, we have immediately
\begin{equation} \label{max-ap1} \sup\limits_{x\in Q}|v_1(x)|\le C\, \int\limits_{Q}\frac{Mg(y)}{|x-y|^{n-\al}}\,dy\le C\,\biggl(\int\limits_{Q}\bigl[Mg(y)\bigr]^p\,dy\biggr)^{\frac1p} r^{\al-\frac{n}p}
\end{equation}
as required. Further, for $v_2$ from Lemma~\ref{lem-a4} we obtain
\begin{equation} \label{a-est25}
\diam \bigl[v_2(Q)\,\bigr]\le C\,r^{1-n}\int\limits_{Q}I_{\al-1}|g|(y)\,dy.
\end{equation}
Thus under assumptions~(\ref{a-est24}) we have 
\begin{equation} \label{a-est26}
\diam \bigl[v(Q)\,\bigr]\le C\,\biggl(\|Mg\|_{L_p(Q)}\,r^{\al-\frac{n}p}+r^{1-n}\int\limits_{Q}I_{\al-1}|g|(y)\,dy\biggr)\end{equation}
as required in Lemma~\ref{est-potent1}. The remaining assertions of Lemmas~\ref{est-potent2}
and  \ref{est-crit-lor} can be proved in the same way with the following modification:
in the proof of Lemma~\ref{est-crit-lor} one has to use the generalised Holder inequality for Lorentz norms 
$$\int\limits_Q\frac{g(y)}{|y-x|^{n-\al}}\,dy\le \| g\|_{L_{p,1}}\cdot\biggl\|\frac{1_Q}{|\cdot-x|^{n-\al}}
\biggr\|_{L_{\frac{p}{p-1},\infty}}= C\, \|g\|_{L_{p,1}}$$
(see, e.g., \cite[Theorem~3.7]{Maly2}\,).

\noindent Dipartimento di Matematica e Fisica Universitа della
Campania "Luigi Vanvitelli," viale Lincoln 5,
81100, Caserta, Italy\\
e-mail: {\it Adele.FERONE@unicampania.it}

\bigskip

\noindent Sobolev Institute of Mathematics, Acad.~Koptyuga pr., 4,
and Novosibirsk State  University, Pirogova str.1, Novosibirsk, 630090, Russia \\
e-mail: {\it korob@math.nsc.ru}

\bigskip

\noindent
 Dipartimento di Matematica e Fisica Universitа della
Campania "Luigi Vanvitelli," viale Lincoln 5,
81100, Caserta, Italy\\
e-mail: {\it albaroviello@msn.com}

\end{document}